\documentclass[11pt]{amsart}%%%%%,leqno
\usepackage{amsmath,amssymb, graphicx, amscd,latexsym,color,comment}

\makeatletter
\newtheorem{Theorem}{Theorem}%[section]
\newtheorem{Lemma}[Theorem]{Lemma}
\newtheorem{Corollary}[Theorem]{Corollary}
\newtheorem{Proposition}[Theorem]{Proposition}

\newtheorem{Remark}[Theorem]{Remark}

\newtheorem{Assertion}[Theorem]{Assertion}

\newcommand{\eps}{\varepsilon}

\newcommand\la{\lambda}
\newcommand\vphi{\varphi}

\newcommand\al{\alpha}
\newcommand\La{\Lambda}
\newcommand\si{\sigma}
\newcommand\be{\beta}
\newcommand\Si{\Sigma}

\newcommand\Ga{\Gamma}
\newcommand\de{\delta}
\newcommand\De{\Delta}

\newcommand\nl{\newline}

\newcommand\Ker{\rm{Ker}\/}

\newcommand\id{\rm{id}}

\newcommand\modulo{\rm{modulo}\/}

\newcommand\Cone{\rm{Cone}\/}

\newcommand\Vol{\rm{Vol}\/}

\newcommand\lcm{\rm{lcm}\/}

\newcommand\inv{^{-1}}

\def\mapright#1{\smash{\mathop{\longrightarrow}\limits^{{#1}}}}%\def\maprightt#1#2{\smash{\mathop{\longrightarrow}\limits^{#1}}}
\def\inv{^{-1}}

\begin{document}

\title[On Zariski-pairs of links]
 %of plane curves and non-reduced degeneration
%{  Draft: \today}
{On $\mu$-Zariski pairs of links}
\author
%Normally smooth divisors  {  Draft: \today}
[M. Oka ]
{Mutsuo Oka} 
\address{
%Department of Mathematics, 
%Tokyo  University of Science,
%Kagurazaka 1-3, Shinjuku-ku,
Nakaochiai 3-19-8,
%Tokyo 162-8601}
Tokyo 161-0032}
 \email{okamutsuo@gmail.com}

%\today
%\thanks{}
%\keywords{32S55,14J17}

\begin{abstract}
The notion of Zariski pairs  for projective curves in $\mathbb P^2$ 
is known since the pioneer paper of Zariski \cite{Zariski} and  for further development, we refer the reference in  \cite{Bartolo}.
In this paper, we introduce a notion of Zariski pair of links in the class of isolated hypersurface singularities.  Such a pair is canonically produced from a Zariski  (or a weak Zariski ) pair of curves $C=\{f(x,y,z)=0\}$ and $C'=\{g(x,y,z)=0\}$ of degree $d$ by simply adding a monomial $z^{d+m}$
to $f$ and $g$
so that the corresponding affine hypersurfaces have isolated singularities at the origin.
They have a same zeta function and a same Milnor number (\cite{Almost}). We give new examples of Zariski pairs which have the same $\mu^*$ sequence and a same zeta function but two functions belong to  different connected components of $\mu$-constant strata (Theorem \ref{mu-zariski}). Two link 3-folds are not diffeomorphic and they are distinguished by the first homology which implies the Jordan form of their monodromies are different (Theorem \ref{main2}).
We start from  weak Zariski pairs of projective curves to construct   new Zariski pairs of  surfaces which have non-diffeomorphic link 3-folds. We also prove that hypersurface pair constructed from a Zariski pair give a diffeomorphic links (Theorem \ref{main3}).
\end{abstract}
\maketitle

\maketitle
\tableofcontents
 
\section{Introduction}
Consider a germ of an  analytic functions $f(\mathbf z)$ defined in a neighborhood $\mathcal U_0$ of the origin  in $\mathbb C^n$  with an  isolated singularity at  $\mathbf z=0$. %the origin.
The {\em $\mu^*$-invariant} of $f$ is defined as  $n$-tuple of integers $\mu^*=(\mu^{(n)},\dots, \mu^{(1)})$ where $\mu^{(i)}$ is the Milnor number of $f|{L_i}$ with $L_i$ being a generic $i$-dimensional linear subspace through the origin (\cite{Tessier}). 
%In \S 3, we introduce a weker pair called {\em a weak Zariki pair}.
%When we consider a weak Zariski pair, the curves $C$ and $C'$ are  not irreducible.
%In this paper, we only consider a weak Zariski pair.

%We now introduce the concept of  a Zariski pair for links of hypersurfaces with isolated singularities at the origin.
Consider a  Zariski  pair  (or  a weak Zariski pair)
 of projective curves  $C, C'$ of degree $d$ defined by homogeneous polynomials $f_d(x,y,z)$ and $g_d(x,y,z)$. We assume that the singular points of $C$ and $C'$ are Newton non-degenerate. Consider the affine hypersurfaces  defined by 
$f(x,y,z)=f_d(x,y,z)+z^{d+m}$ and $g(x,y,z)=g_d(x,y,z)+z^{d+m}$. Then $f$ and $g$ are almost non-degenerate function in the sense of \cite{Almost} and two hypersurfaces have a same zeta function. In \cite{Almost}, we called such a pair {\em a Zariski pair of links}. It is also easy to observe that $f$ and $g$ have a same $\mu^*$-sequence.
Consider the local links at the origin,  $K_f=V(f)\cap S_{\eps}^{2n-1}$ and $K_g=V(g)\cap S_\eps^{2n-1}$, with a sufficiently small 
$\eps\ll 1$. We say $\{K_f,K_g\}$ (or $\{f, g\}$) is a {\em $\mu$-Zariski pair of links} (or {\em of surfaces }), 
if $f$ and $g$  belong to   different connected components of  $\mu$-constant strata. For the definition of $\mu$-constant strata, see \cite{EO}.
We say    $\{K_f, K_g\}$ are {\em $\mu^*$-Zariski pair} if  
%they have also the same $\mu^*$ invariants but 
$f$ and $g$ belong
to different connected components of $\mu^*$-constant strata. 
In \cite{EO}, we gave an example of $\mu^*$-Zariski pair of links.
Here $\mu^*$-invariant  is introduced by Teissier. 
%$\mu^{(i)}$ is the Milnor number of $f$ restricted to a generic $i$-planes through the origin.

In \cite{Tessier}, Teissier proved that if $f$ and $g$ 
%belongs to the same connected component of $\mu^*$-constant stratum and they
 are connected by a complex  piecewise analytic $\mu^*$-constant  family
 \footnote{this means, there exists a finite numbers $t_0=0<t_1<\dots<t_k=1$ and open neighborhood $U_i$ of $[t_i,t_{i+1}]$ in $\mathbb C$ and an analytic family of functions $ f_{it}(\mathbf z), t\in U_i$ extending $f_t|_{[t_i,t_{i+1}]}$
 
 }
  $f_t, 0\le t\le 1$ with $f_0=0$ and $f_1=g$, then the local link pairs
$(S_\eps^{2n-1},K_f)$ and $(S_\eps^{2n-1},K_g)$ are diffeomorphic. The same assertion is true for $\mu$-constant $C^{\infty}$-real 
family $f_t$ for $n\ne 3$ by L\^e and Ramanujam \cite{Le-Ramanujam}.  

The purpose of this paper is to present some examples of $\mu$-Zariski pairs of surfaces in \S 3 (Theorem \ref{mu-zariski}). We will show that those pairs of surface links constructed from weak Zariski pairs in \S 3   are not diffeomorphic. 
On the other hand, the link pairs constructed from  Zariski pairs are always diffeomorphic as 3-manifolds (Theorem \ref{main3}) but we do not know if this diffeomorphism can be extended to a diffeomorphism of $S^5$. %%%%%%

\section{Preliminary}
\subsection{Divisor of rational functions}
Consider a rational function $\vphi(t)=p(t)/q(t)$ where $p(t),q(t)\in \mathbb C[t]$ and  $p(0),q(0)\ne 0$ and consider the factorization  
$p(t)=c\prod_{i=1}^\ell(t-\al_i)^{\nu_i}, c\ne 0$ and $q(t)=c'\prod_{j=1}^m (t-\be_j)^{\mu_j},\,c'\ne 0$. The divisor of $\vphi$ is defined  by 
\[
\rm{div} ( \vphi)=\sum_{i=1}^\ell \nu_i<\al_i>-\sum_{j=1}^m
\mu_j<\be_j>\,\in \mathbb Z\mathbb C^*
\]
where $\mathbb Z\mathbb C^*$  is 
the group ring of $\mathbb C^*$. We denote the divisor of $(t^d-1)$ by $\La_d$.
The degree of $\vphi$ is defined by $\deg\,\vphi=\deg\,p-\deg\, q=\ell-m$.
The following formula (\cite{Orlik-Milnor}) is  useful in the later discussions:
\begin{eqnarray}\label{product-formula}
\La_d\cdot\La_{d'}=\gcd(d,d')\La_{\lcm(d,d')}.
\end{eqnarray}
%%%%%%
\subsection{Zeta function of the Milnor fibration}\label{zeta function}
For an analytic function $f(\mathbf z)$ defined in a neighborhood of the origin, we consider the tubular Milnor fibration  %(\cite{Milnor})
%$\vphi: S_\eps^{2n-1}-K_f\to S^1$ 
$f: E(\eps,\de)^*\to D_\de^*$
where 
\[\begin{split}
E(\eps,\de)^*&=\{\mathbf z\in\mathbb C^n\,|\, \|\mathbf z\|\le \eps,\,0\ne |f(\mathbf z)|\le \de\},\\
D_\de^*&=\{\eta\in \mathbf C\,|\, 0<|\eta|\le \de,\,\,\de\ll \eps\ll 1.
\end{split}
\]
%$S_\eps^{2n-1}=\{\mathbf z\in \mathbb C^n\,|\, \|\mathbf z\|=\eps\}$ 
Let $F$ be the fiber and let $h: F\to F$ be the monodromy map. Consider the characteristic polynomial $P_j(t)=\det({\id}-t\, h_{*j})$  where  $h_{*j}:H_j(F;\mathbb Q)\to H_j(F;\mathbb Q)$.
 The zeta function of the Milnor fibration, denoted    as $\zeta_f(t)$ is defined by the alternative product of the characteristic polynomials
 $\zeta_f(t)=P_0(t)^{-1}P_1(t)\cdots P_{n-1}(t)^{(-1)^n}$. If $f$ has an isolated singularity at the origin, $F$ is (n-2)-connected
 and 
 \[
 \zeta_f(t)=P_{n-1}(t)^{(-1)^n}(1-t)^{-1},\quad \deg\,\zeta_f=-1+(-1)^n \mu
 \]
  where $\mu$ is the (n-1)-th
 Betti number of $F$ and  $\mu$ is usually  called the Milnor number of $f$ at $0$ (\cite{Milnor}).
\subsection{A'Campo  formula}
Consider an analytic function $f(\mathbf z)=\sum_{\nu}a_\nu\mathbf z^{\nu}$ of $n$ variables defined in a neighborhood of the origin of $\mathbb C^n$.  
Assume that we are given a good resolution $ \hat\pi:X\to  \mathcal U_0$ of the function $f$ and let $E_1,\dots, E_s$ be the exceptional divisors of $ \hat\pi$,
that is    $\hat\pi\inv(V)= \widetilde V\cup_{i=1}^s E_j$ where $\tilde V$ is the strict transform of the hypersurface $V:=f\inv(0)$ and  $\mathcal U_0$ is a small neighborhood of the origin.
The irreducible components of  $\tilde V$ and $E_j,\,j=1,\dots,s$ are non-singular and
$ \hat\pi\inv(V)$ has only ordinary normal crossing singularities.
%$\tilde V$ is non-singular and $\hat\pi f\inv(0)$ has only normally crossing singularities.
Consider the open  subset $E_j'=E_j\setminus( \widetilde V\cup_{i\ne j}E_i)$ and $E_j''=E_j'\cap\hat\pi\inv(0)$.  In particular, if $E_j$ is a compact divisor, 
$E_j''=E_j'$.
%where $\widetilde V$ is the strict transform of the hypersurface $V=f\inv(0)$. 
Let $m_j$ be the multiplicity of $ \hat\pi^*f$ along $E_j$.
\begin{Lemma} [ A'Campo \cite{ACampo}]
The zeta function of the Milnor monodromy at the origin is given by the formula:
\begin{eqnarray}\label{AC}
\zeta(t)=\prod_{j=1}^s(1-t^{m_j})^{-\chi(E_j'')}.\hspace{2cm}\notag
\end{eqnarray}
\end{Lemma}
In this formula, the singularities at the origin can be non-isolated. As a simple corollary of the A'Campo formula , we have
\begin{Corollary}\label{uniqueness}
The divisor of the  zeta function of Milnor monodromy is uniquely expressed as 
${\rm{div}}\,\zeta(t)=\sum_{i=1}^s \nu_i\La_{d_i}$ with $d_1<\dots<d_s$ and $\nu_i\ne 0$ for $i=1,\dots,s$.
\end{Corollary}
  \subsection{ Newton boundary and  dual Newton diagram}
  \subsubsection{Newton boundary}
  Let $f(\mathbf z)=\sum_\nu a_\nu \mathbf z^\nu$ be a given  holomorphic function defined by a convergent power series. 
  Let $M$ be the space of monomials of the fixed coordinate variables $z_1,\dots, z_n$ of $\mathbb C^n$ and let $N$ be the space of weights  of the  variables $z_1, \dots, z_n$. 
  We identify   the monomial $\mathbf z^{\nu}=z_1^{\nu_{1}}\dots z_n^{\nu_n}$ and the integral point $\nu=(\nu_{1},\dots, \nu_n)\in \mathbb R^n$. A weight $P$  is also identified with
  the column vector ${}^t(p_1,\dots, p_n)\in \mathbb R^n$ where  $p_i=\deg_P (z_i)$ and we call $P$ a weight vector.
The  {\em Newton polygon} $\Ga^+(f)$ with respect to the given coordinates $\mathbf z=(z_1,\dots, z_n)$ is the convex hull of the union $\cup_{\nu, a_\nu\ne 0} \{\nu+\mathbb R_{\ge 0}^n\}$ and {\em the Newton boundary} $\Ga(f)$  is defined  by the union of compact faces of $\Ga^+(f)$.
  For a non-negative   weight vector $P={}^t(p_1,\dots, p_n)$, we consider the canonical linear function $\ell_P$ on $\Ga^+(f)$
 % $\ell_P:\Ga^+(f)\to \mathbb R_+$ 
  which is defined by $\ell_P(\nu)=\sum_{i=1}^n \nu_ip_i$.
  This is nothing but  the degree mapping $\deg_P \mathbf z^\nu=\ell_P(\nu)$.
   The minimal value of $\ell_P$ is denoted by
  $d(P; f)$.  Put $\De(P; f):=\{\nu\in \Ga^+(f)\,|\, \ell_P(\nu)=d(P)\}$. We will  use the simplified notations $d(P)$ and $\De(P)$ if any ambiguity seems unlikely.  In general,
  $\De(P)$ is a face of $\Ga^+(f)$ and $\De(P)\subset \Ga(f)$ if $P$ is {\it   positive} (i.e., $p_i>0,\forall i$).
For a maximal dimensional face $\De$, i.e.   $\De\subset \Ga(f)$ with $\dim\, \De=n-1$, there is a unique   positive primitive integer vector $P$ such that $\De(P)=\De$.
%(Recall that $P={}^t(p_1,\dots, p_n)$ is primitive if $\gcd\,\{p_1,\dots,p_n\}=1$.)
  The  partial sum $\sum_{\nu\in \De}a_\nu \mathbf z^\nu$ is called {\it the face function} of the face  $\De$ and we denote it  as 
  $f_{\De}$.  For a weight $P$, $f_P$ is defined by $f_{\De(P)}$. Note that $f_P$  is a polynomial if $P$ is   positive.
  \begin{Remark}{\rm 
  In this paper,  we used the terminologies for a weight vector {\em  positive } and %
{\em  non-negetive} instead of {\em  strictly  positive } and   {\em positive}  weight vectors respectively, terminologies used  in \cite{Okabook,Almost}.
We changed the terminologies so that it is consistent with  our paper \cite{EO}.}
  \end{Remark}
  \subsubsection{Dual Newton diagram}
  Two weight vectors $P,Q$ are equivalent if and only if $\De(P)=\De(Q)$ and this equivalent relation gives a conical subdivision of  the space of the non-negative weight vectors $N^+_{\mathbb R}$, i.e.  of $\mathbb R_{\ge 0}^n$
  (under the above identification) and we denote it as $\Ga^*(f)$ and call it {\em the dual Newton diagram of $f$}. 
  We say, $f$ is {\it Newton non-degenerate
  on  $\De\subset \Ga(f)$} if $f_\De:\mathbb C^{*n}\to \mathbb C$ has no critical points. We say $f$ is {\it Newton non-degenerate} if it is non-degenerate on every face $\De\subset \Ga(f)$ of any dimension.  The dimension of a face can be any non-negative integer less than $n$.
  The closure of an  equivalent class can be expressed  as 
  \[
  {\Cone}(P_1,\dots, P_k):=\left\{\sum_{i=1}^k \la_i P_i\,|\, \la_i\ge 0\right \}\]
  where  $P_1,\dots, P_k$ are chosen to be primitive integer vectors. This expression is unique if $k$ is minimal among any possible such expressions.
  A cone $\si={\Cone}\,(P_1,\dots, P_k)$ is {\it simplicial} if  $\dim\,\si=k$ and $\si$ is {\it regular} if $P_1,\dots, P_k$ are primitive integer vectors which can be  extended to a basis of the lattice $\mathbb Z^n\subset \mathbb R^n$.
Recall that $f$ is {\em convenient} if $\Ga(f)$ touches with every coordinate axis. We say   $f$ is {\em pseudo-convenient} if $f$ is written as $f(\mathbf z)=\mathbf z^{\nu_0} f'(\mathbf z)$
where $f'$ is a convenient analytic function and $\nu_0$ is a non-negative integer vector. In this case,
$\Ga^*(f)=\Ga^*(f')$.
%%%%%
 \subsection{Toric modification } %and the associated  toric stratification}
% In this subsection,
%  we use the same notation as \cite{Almost}.
   %\subsection{Toric modification}
  A regular simplicial cone subdivision $\Si^*$ of the space of non-negative weight vectors $N^+_{\mathbb R}=\mathbb R_+^n$ is  {\it admissible with the dual Newton diagram $\Ga^*(f)$} if $\Si^*$ is a subdivision of $\Ga^*(f)$. For such a regular simplicial cone subdivision, we associate a modification 
  $\hat\pi:X\to  \mathbb C^n$ as follows:
  let $\mathcal S_n$ be the set of $n$-dimensional cones in $\Si^*$.
  For each $\si={\Cone} (P_1,\dots, P_n)\in \mathcal S_n$, we identify $\si$ with the unimodular matrix:
  \[
  \si=\left(\begin{matrix}
  p_{11},&\dots&p_{1n}\\
  \vdots&\vdots&\vdots\\
  p_{n1}&\dots&p_{nn}\end{matrix}
  \right)
  \]
  with $P_j={}^t(p_{1j},\dots, p_{nj})$. 
  For a  unimodular matrix $A=(a_{ij})_{1\le i,j\le n}$, we define  an isomorphism of the torus  $\psi_A:\mathbb C^{*n}\to \mathbb C^{*n}$ by $\mathbf w=\psi_A(\mathbf z), w_i=z_1^{a_{i1}}\dots z_n^{a_{in}},\,i=1,\dots,n$.
  To each $\si\in \mathcal S_n$, we associate an affine coordinate chart  $(\mathbb C_\si^n,\mathbf u_\si)$ with coordinates $\mathbf u_\si=(u_{\si 1},\dots, u_{\si n})$. The modification $\hat \pi$  is defined as follows. For each $\si\in \mathcal S_n$,
  we associate a birational mapping $ \hat\pi_\si=\psi_\si:\mathbb C_\si^n\to \mathbb C^n$ by $z_i=u_{\si 1}^{p_{i1}}\dots u_{\si n}^{p_{in}}$
  for $i=1,\dots,n$ 
  under the identification of $\si$ and the above  unimodular matrix. Then  a complex manifold $X$ is constructed  by gluing $\mathbb C_\si^n$ and $\mathbb C_\tau^n$ by $ \hat\pi_{\tau}\inv\circ\hat\pi_\si:\mathbb C_\si^n\to \mathbb C_\tau^n$ where it is well-defined.  This defines a proper modification $ \hat\pi:X\to  \mathbb C^n$ ( for further detail, see  Theorem (1.4), Chapter II,\cite{Okabook}).
 
   \subsubsection{Exceptional divisors corresponding to vertices of $\Si^*$}
 Suppose $\si={\Cone} (P_1,\dots, P_n)$ and $\tau={\Cone}(Q_1,\dots, Q_n)$ have a same vertex  $Q_i=P_1$ for some $i$. Taking a permutation of the vertices, we may always assume that $i=1$. 
The hyperplane $\{u_{\si 1}=0\}\subset \mathbb  C_\si^n$ glues  canonically  with the hyperplane $\{u_{\tau 1}=0\}\subset \mathbb C_\tau^n$. Thus 
 any vertex\footnote{
 A vertex  is a   primitive integral vector which generate a one-dimensional cone of $\Si^*$. } $P$ of $\Si^*$, gluing the hyperplanes on  every such toric coordinates with $P_1=P$,   
 defines a rational divisor in $X$,  and we denote this divisor by $\hat E(P)$. We say  vertices $Q_i, i=1,\dots,r$  
 are {\em adjacent} if $\Cone\,(Q_1,\dots, Q_r)$ is a cone in $\Si^*$. By the  assumption that $\Si^*$ is admissible, $\cap_{i=1}^r\hat E(Q_i)$ is non-empty if and only if $Q_i,\dots, Q_r$ are adjacent (Proposition (1.3.2), Chapter II, \cite{Okabook}).
If $P$ is   positive, $\hat E(P)$ is a compact divisor and $\hat\pi(\hat E(P))=\{O\}$. 
%Recall that  a vertex of $\Si^*$ is a primitive integer generator of a 1-dimensional cone of $\Si^*$.
Let $\mathcal V^+$ be the set of  non-negative vertices $P$ of $ \Si^*$ with $d(P)>0$ and $P\ne e_1,\dots, e_n$.  A germ of a function $f(\mathbf z)\in \mathcal O_0$ is called monomial-factor free if the factorization of $f$ does not have any monomial factor.
 For a monomial-factor free $f$, $\Si^*$ is called {\em small} if $\mathbb C^I$ is a  not vanishing coordinate subspace of $f$, then $e_{I^c}:=\Cone\{e_j\,|\, j\in I^c\}$ is a cone in $\Si^*$ where $I^c=\{1\le j\le n\, |j\notin I\}$. In the  case $f$ being not monomial-factor free, 
write  $f=Mf'$ where $M$ is a monomial and $f'$ is monomial-factor free, and 
$\Si^*$  is small for $f$ if it is small for $f'$. Here we note that $\Ga^*(f)=\Ga^*(f')$.
% if for any vertex  $P$ of $ \Si^*$ other than $e_1,\dots, e_n$,  $d(P)>0$. %and  $\mathcal V^+$ is the set of such vertices.
 %See Chapter III, \cite{Okabook}. 
%$f(\mathbf z)$ is a {\em pseudo-convenient} if $f(\mathbf z)$ is a product of a convenient polynomial and a monomial.
Note that if $f$ is pseudo-convenient, in a small regular simplicial cone subdivision $\Si^*$, every
vertices are  positive except for the canonical generators  $\{e_1,\dots, e_n\}$ of the lattice $\mathbb Z^n$ where 
$e_i={}^t(0,\dots,\overset{\overset i{\smile}}1,\dots,0)$.
% We call such a regular simplicial cone subdivision 
If $\Si^*$  is {\em small},  the associated modification $\hat\pi:X\to   \mathbb C^n$ is called a {\em small } toric modification hereafter. If $\hat\pi$ is small and $f$ is pseudo-convenient,   $\hat E(e_i)^*$ is surjectively mapped onto $\{z_i=0,z_j\ne 0,j\ne i\}$.
See \S \ref{toric-stratification} for the definition of $\hat E(e_i)^*$.
In this paper, we consider  functions which are either convenient or at most the pseudo convenient
and we assume that  $\Si^*$ is small. Thus 
any vertex $P\in \mathcal V^+$ is, in fact,   positive and 
 $\hat\pi\inv(O)=\cup_{P\in \mathcal V^+}\hat E(P)$.
%For $P\in \mathcal V^+$, corresponding exceptional divisor, noted as $\hat E(P)$ is defined as 
%$u_{\si 1}=0$ in a toric chart $\mathbb C_\si^n$ with $\si={\Cone} (P,P_2,\dots, P_n)$. 
\subsubsection{Pull-back of $f$}
We consider  hypersurface $V:=f\inv(0)$ and take a sufficiently small neighborhood $\mathcal U_0$ where $f$ is defined. In the following, we restrict $\hat \pi$ over $\hat\pi\inv(\mathcal U_0)$, whenever we consider the strict transform of $V$ in the upper space $X$.
The pull-back   ${\hat\pi}^*f$ of $f$ is expressed in the toric chart $\mathbb C_\si^n$ with $\si={\Cone}(P_1,\dots, P_n)$ as  follows:
\[
 {\hat\pi}^*f(\mathbf u_\si)=\left(\prod_{i=1}^n u_{\si,i}^{d(P_i)}\right)\widetilde f(\mathbf u_\si)
\]
and $\widetilde f(\mathbf u_\si)$ is the defining function of the strict transform $\widetilde V$ of $V$.  The intersection $E(P):=\widetilde V\cap\hat E(P)$ is 
a divisor in $\hat E(P)$ and it is defined  by $g(u_{\si 2},\dots,u_{\si n}):=\tilde f(0,u_{\si 2},\dots, u_{\si n})=0$. (Recall we have assumed $P=P_1$.) More explicitly, we have
\[
g(u_{\si 2},\dots,u_{\si n})=f_{P}(\hat\pi_\si(\mathbf u_\si))/\left(\prod_{i=1}^n u_{\si,i}^{d(P_i)}\right)
\]
where $f_{P}$ is the face function of $f$ with respect to $P=P_1$.
$E(P)\subset \tilde V$ is an exceptional divisor of the restriction $\pi:=\hat\pi|_{\widetilde V}:\widetilde V\to V$
and $E(P)$ is non-empty if and only if $\dim\, \De(P)\ge 1$.
%%%%%
\subsubsection{Toric stratification}\label{toric-stratification}
Let $P$ be a vertex of $\Si^*$ and let $\mathcal C(P)$ be the set of cones $\tau=\Cone\,(P_1,\dots, P_k)\in \Si^*$ with
$P_1=P$.  Choose a maximal cone $\si=\Cone\,(P_1,\dots, P_n)$ which  has  $\tau$  as a boundary face.   
Put 
\[\begin{split}
\hat E(\tau)^*&:=\cap_{i=1}^k\hat E(P_i)\setminus\cup_{j=k+1}^n\hat E(P_j)\\
&=\{\mathbf u_\si\in \mathbb C_\si^n\,|\, u_{\si 1}=\cdots=u_{\si k}=0,\,u_{\si j}\ne 0,\,\, j>k+1
\}\\
&\cong \mathbb C^{*(n-k)}.
\end{split}
\]
$\hat E(\tau)^*$ does not depend on the choice of $\si$ and 
$\hat E(\tau)^*$ is isomorphic to the torus $\mathbb C^{*(n-k)}$.
Now we see that  $\hat E(P)$ has a disjoint partition
$\amalg_{\tau\in \mathcal C(P)}\hat E(\tau)^*$ which we call {\em the toric stratification} of $\hat E(P)$.
In particular, we put $\hat E(P)^*=\hat E(P)\setminus \cup_{Q\in \mathcal V^+,Q\ne P}\hat E(Q)$. This is the maximal dimensional torus in $\hat E(P)$.
Let $I\subset \{1,\dots, n\}$ and put $e_{I}$ be the cone generated by $\{e_i\,|\, i\in I\}$.
Let $\hat E(P)_I^{*}:=\hat E(P)\cap \hat E(e_{I^c})^*$, which is  empty if
 vertices $\{P, e_j\,|\,\,j\in I^c\}$ are not adjacent and put  $E(P)_I^{*}=E(P)\cap  \hat E(P)_I^{*}$.
Here $I^c=\{1,\dots,n\}\setminus I$.
Then we use  the toric decomposition of 
\nl
$\hat E'(P):=\hat E(P)\setminus \tilde V\cup_{Q\ne P} \hat E(Q)$ as 
\[
\hat E(P)'=\amalg_I\hat E(P)_I^{'},\quad \hat E(P)_I'=\hat E(P)_I^*\setminus E(P)_I^*. 
\]
  Take a vertex $P\in \mathcal V^+$  and we consider  the exceptional divisor $\hat E(P)$.
 We compute the Euler characteristic $\chi(\hat E(P)')$ in the A'Campo formula using the toric stratification.
 Let $\hat E(P)_I'=\hat E(P)_I^*\setminus E(P)_I^*$.
 Note that  $\hat E(P)_I'$ is non-empty only if $\{P,e_i|\, i\notin I\}$ spans a simplicial cone in $\Si^*$. For $I=\{1,\dots, n\}$, we omit
the suffix $I$. By the additivity of Euler characteristics, we have
 \begin{eqnarray}\label{Acampo-degenerate}
 \hat E(P)'&=&\amalg_{I}(\hat E(P)_I^*\setminus E(P)_I^*)=\amalg_I \hat E(P)_I',
 \notag\\
 \chi( \hat E(P)')&=&-\sum_I \chi( E(P)_I^*). %\notag
 \end{eqnarray}
In the following argument, we use  the additivity of Euler characteristics and the decomposition (\ref{Acampo-degenerate}) which is also valid for a function which has some Newton  degenerate faces.
 So consider function $f$ which has some degenerate faces like 
 almost Newton non-degenerate functions which we recall in \S \ref{almost-section}. First we take an admissible toric modification $\hat\pi:X\to   \mathbb C^n$. The strict transform $\tilde V$ or $\hat\pi^*f\inv(0)$ has still some singularities.
Take further modification $\omega:Y\to X$ so that $\Pi:=\hat\pi\circ\omega:Y\to  \mathbb C^n$ 
is a good resolution of $f$ when it is restricted over $\mathcal U_0$. Let $D_1,\dots, D_s$ be the exceptional divisors by 
$\omega$, let $\hat V$ be the strict transform of $V=V(f)$ to $Y$ 
and let $d_j$ be the multiplicity of $\Pi^*f$ along $D_j$.
 The A'Campo formula can be expressed as 
 \begin{eqnarray}(AC')
\qquad \zeta(t)&=\zeta_\omega(t)\prod_{P\in \mathcal V^+}(1-t^{d(P)})^{-\chi(\hat E(P)')}\notag \\
&=\zeta_\omega(t)\prod_{P\in \mathcal V^+}\prod_I(1-t^{d(P)})^{\chi( E(P)_I^*)}.\notag
 \end{eqnarray}
 where $\zeta_\omega(t)=\prod_{j=1}^s(1-t^{d_j})^{-\chi(D_j')}$.
 See Lemma \ref{modified-Varchenko} below for detail.
 %%%
 
% \begin{comment}
\subsubsection{Kouchnirenko formula}
Consider  a polynomial $h(\mathbf y)=\sum_{i=1}^s a_i\mathbf y^{\nu_i}\in \mathbb C[y_1,\dots,y_m]$ of $m$-variables $\mathbf y=(y_1,\dots, y_m)$ where $a_i\ne0$ for $i=1,\dots, s$.  The Newton diagram  $\De(h)$ of $h$ is defined by the convex hull of $\{\nu_i\,|i=1,\dots, s\}$.
Note that $\De(h)$ is a compact polyhedron.
We say $h$ is {\em Newton non-degenerate} if $V(h_\De)^*:=h_\De\inv(0)\cap \mathbb C^{*m}$ has no singular point for every face $\De$ of $\De(h)$ (including $\De(h)$).
A key observation for the calculation is 
\begin{Lemma}[Kouchnirenko \cite{Ko}, Oka \cite{Ok3}] \label{keyLemma}
Assume that $h(\mathbf y)\in \mathbb C[y_1,\dots,y_m]$ is a Newton non-degenerate    polynomial 
and let
$V(h)^*=\{\mathbf y\in \mathbb C^{*m}\,|\, h(\mathbf y)=0\}$.
Then the Euler characteristic  is given as 
\[ \chi(V(h)^*)=(-1)^{m-1}m! {\Vol}_m \De(h).
\]
\end{Lemma}
In particular, if $\dim\,\De(h)<m$, $\chi(V(h)^*)=0$.
We use also the following vanishing property.
 \begin{Proposition} \label{vanishing}
 Assume that $h(\mathbf y)$ is an arbitrary (not necessary Newton non-degenerate) polynomial such that $\dim\, \De(h)<m$. 
 %Let $V(h)^*=\{\mathbf y\in \mathbb C^{*m}\,|\, h(\mathbf y)=0\}$.
 Then $\chi(V(h)^*)=0$.
 \end{Proposition}
 \begin{proof}  Let $r=\dim\, \De(h)$. We can take a unimodular matrix $A$ so that  after change of variables by $\mathbf y=\pi_A(\mathbf x)$, we can write 
 $ h(\pi_A(\mathbf x))=M h'(x_1,\dots, x_r)$ where $M$ is a monomial of $x_1,\dots, x_m$ and $h'$ is a polynomial of $r$-variables $x_1,\dots, x_r$.
As $V(h')^*$ is a product  $(V(h')^*\cap\mathbb C^{*r})\times \mathbb C^{*(m-r)}$, the Euler characteristic is zero.
% where $h''$ is the polynomial $h'$ considered a function of $x_1,\dots, x_r$.
As 
$V(h')^*$ is isomorphic to $V(h)^*$, the assertion follows from this expression.
 \end{proof}

%\end{comment}
%%%%%%
\subsection{Varchenko formula}
Suppose that $f(\mathbf z)$ is Newton non-degenerate. 
For each non-vanishing coordinate subspace $\mathbb C^I$,
$I\subset \{1,\dots,n\}$, let  $\mathcal P_I$ be the set of primitive integer weight vectors of 
 $f^I$
%$\mathbb Z_{\ge 0}^I$
 which correspond to the maximal dimensional faces of $\Ga(f^I)$.  
Using a toric modification
$ \hat\pi:X\to   \mathbb C^n$ which is  admissible with the dual Newton diagram $\Ga^*(f)$,  the integer
$\chi(\hat E(P)')$ can be computed combinatorially.
Namely the zeta function is given  as 
\begin{Lemma}[Varchenko \cite{Va}]\label{Va}
Suppose that $f(\mathbf z)$ is Newton non-degenerate.  Then
\begin{eqnarray} \label{Va1}
  \zeta(t)&=\prod_{I}\zeta_I(t),\quad \zeta_I(t)=\prod_{Q\in \mathcal P_I}(1-t^{d(Q;f^I)})^{\chi(E(Q)^*)}%\notag
\end{eqnarray}
where  $f^I$ is the restriction of $f$ to the coordinate subspace $\mathbb C^I$.
\end{Lemma}
%For $I=\{1,\dots, n\}$, we simply denote $\mathcal P$.
The number $\chi(E(Q)^*)$ satisfies the following equality, if $f$ is Newton non-degenerate.
\begin{Proposition}Suppose that $f(\mathbf z)$ is Newton non-degenerate. Then the above integer satisfies the equality:
\begin{eqnarray}
\label{Va2}
 \chi(E(Q)^*)&=&(-1)^{|I|}|I|!{\Vol}_{|I|}{\Cone}(\De(Q;f^I) )/d(Q;f^I).
%\end{split}
\end{eqnarray}
\end{Proposition}

\begin{Remark}{\rm The vertices in $\mathcal P_I$ are used to compute the zeta function but the vertices of  $\mathcal P_I$ are not in $\mathcal V^+$. Thus they dot not appear in A'Campo formula. The correspondence of A'Campo formula and the Varchenko formula is explained by Lemma \ref{keyLemma}
and the following observation: for any $Q\in \mathcal P_I$, there is a unique vertex $P\in \mathcal V^+$ such that 
$P$ and $\{e_i,\,i\notin I\}$ span a simplex of $\Si^*$ and $\De(P;f)\cap \mathbb R^I=\De(Q;f^I)$
and $d(P)=d(Q)$. See \S 5, Chapter III, \cite{Okabook} for further discussion.
}
\end{Remark}

\subsection{Almost non-degenerate functions}\label{almost-section}
Consider a  convenient analytic function $f(\mathbf z)=\sum_{\nu} a_\nu {\mathbf z}^\nu$ which is expanded in a Taylor series and let $\Gamma(f)$ be the Newton boundary.
Let $\hat\pi:X\to   \mathbb C^n$ be a toric modification with respect to %a regular simplicial cone subdivision 
 $\Si^*$ which is a small  simplicial regular subdivision $\Si^*$ of 
the dual Newton diagram $\Ga^*(f)$. 
%Hereafter $\mathcal U_0$ is a small neighborhood of the origin.
%Let ${\hat\pi}^*f$ be the pull-back of $f$. % and let $\hat E(P)$
Let $\mathcal M$ be the set of maximal dimensional faces of $\Ga(f)$ and let $\mathcal M_0$ be a given subset of $\mathcal M$ so that 
for $\De\in \mathcal M_0$, $f_\De:\mathbb C^{*n}\to\mathbb C$ is Newton degenerate.
%if and only if $\De\in \mathcal M_0$.
For $\De\in \mathcal M_0$, let  $P_\De$  be the unique vertex   which corresponds to $\De$:  $\De(P_\De)=\De$. Recall that $\hat E(P)$  is an exceptional divisor which corresponds to  the vertex $P$. We consider the following conditions on $f$.
\nl
 (A1) For any face $\De$ of $\Ga(f)$ with either $\De\in \mathcal M\setminus \mathcal M_0$ or $\dim\, \De\le n-2$,
$f$ is Newton non-degenerate on $\De$. For  $\De\in \mathcal  M_0$, $f_{\De}:\mathbb C^{*n}\to \mathbb C$ has a finite number of 1-dimensional critical loci which are $\mathbb C^*$-orbits through the origin. 
%\end{enumerate}
Recall that 
% Take the weight vector vector $P={}^t(p_1,\dots, p_n)$ corresponding to $\De$. Then 
$f_\De(\mathbf z)$ is a weighted homogeneous polynomial with respect to the weight vector $P_\De={}^t(p_1,\dots,p_n)$ and there is an associated $\mathbf C^*$-action defined by $t\circ (z_1,\dots, z_n)=(t^{p_1}z_1,\dots, t^{p_n}z_n),\,t\in \mathbf C^*$.
% where $\De(P)=\De$ and $P={}^t(p_1,\dots,p_n)$. Critical points loci of $f_\De$ are stable under this action.

Let $\si={\Cone} (P_1,\dots, P_n)$ be a simplicial cone in $\Si^*$ such that $P_1=P_\De$.
Let $\mathbf u_\si=(u_{\si 1},\dots, u_{\si n})$ be the corresponding toric coordinate chart.
The strict transform $\widetilde V$ of $V(f)$ is  defined  by $\widetilde f(\mathbf u_\si)=0$ where
$\widetilde f$ is defined by the equality:
\[
\hat f(\mathbf u_\si):={ {\hat\pi}}^*f(\mathbf u_\si)=\left(\prod_{i=1}^n u_{\si,i}^{d(P_i)}\right)\, \widetilde f(\mathbf u_\si)
\]
 and $E(P_1)\subset \hat E_0$ is defined by 
 $\{\mathbf u_\si\,|\,u_{\si 1}=0,g_\De(u_{\si 2},\dots, u_{\si n})=0\}$
 where
  \[\begin{split}
  g_{\De}(u_{\si, 2},\dots, u_{\si n})&:=\widetilde f(0,u_{\si 2},\dots,u_{\si n})\\
  &=f_\De(\pi_\si(\mathbf u))/\prod_{i=1}^n u_{\si,i}^{d(P_i)}.
  \end{split}
  \]
 For simplicity, we denote  $\hat\pi^*f$ by $\hat f$ hereafter.
 The assumption (A1) implies that  $E(P_1)$  as a divisor of $\hat E(P_1)$ has 
 a finite number of  isolated singular points. In fact, this follows from the isomorphism:
 \[
\hat \pi_\si: \,V(g_\De)\cap\mathbb C_\si^{*n}=\mathbb C^*\times (V(g_\De|_{ \{u_{\si 1}=0\}})\cap\mathbb C_\si^{*(n-1)})\to
 V(f_\De)\cap\mathbb C^{*n}.
 \]
  Let  $S(\De)$ be the set of the singular points of $E(P_1)$.
 Take any $q\in S(\De)$.
 % and assume $q=(0,\be_2,\dots, \be_n)$ in $\mathbb C_\si^n$.
An {\em admissible coordinate  chart at $q$}  is  an analytic coordinate chart $(U_q,\mathbf w)$, $\mathbf w=(w_1,\dots,w_n)$ centered at $q$ where $U_q$ is an open neighborhood of $q$ and $(w_2,\dots, w_n)$ is an analytic coordinate change of $(u_{\si 2},\dots, u_{\si n})$ and  $w_1=u_{\si 1}$.  
We say that $f$ is {\em a weakly  almost non-degenerate function}  if it satisfies  $(A1)$.
% (In many cases,  we can take $w_i=u_{\si,i}-\be_i,\,i=2,\dots,n$.)  As $w_1=u_{\si 1}$,   $w_1 =0$ is the defining function of $\hat E(P_1)$.
 As a second condition, we consider
  %\begin{enumerate}
  \nl
 ($A_{2}$)
 {\em For any $\De\in \mathcal M_0$ and $q\in S(\De)$,  there exists an admissible coordinate $(U_q,\mathbf w)$ centered at $q$ such that  
 ${\hat \pi}^*f(\mathbf w)$ is Newton non-degenerate  and pseudo-convenient with respect to this coordinates 
$(U_q,\mathbf w)$.}
%\end{enumerate}

We say that $f$ is  {\em an  almost non-degenerate function} 
 if 
 %there exists a subset $\mathcal M_0$  of $\mathcal M$ so that 
 it satisfies 
 ($A1$) and ($A2$). 
 %(Actually $(A_{2-\eps})$ follows from (A1).)
% \begin{enumerate}
 %\item[(A2)]
For a weakly  almost non-degenerate functions, 
  the following theorem holds.
\begin{Theorem}[\cite{Almost}] \label{main1}
Assume that $f$  is a weakly almost non-degenerate function.
Then the zeta function of $f$ is given by 
\[\zeta(t)=\zeta(t)'\prod_{\De\in \mathcal M_0}\zeta_\De(t)\]
where
%$\zeta'(t)=\zeta^{(s)}(t)\zeta^{er}(t)$ and 
%\end{split}\]
$\zeta'(t)$ is the zeta function of $\hat  f$ outside of the union of $\eps$ balls centered at $q\in S(\De),\,\De\in \mathcal M_0$ and 
$\zeta'(t)=\zeta^{(s)}(t)\zeta^{er}(t)$ where 
  $\zeta^{(s)}(t)$ is the zeta function of the Newton non-degenerate function $f_s$ with $\Ga(f_s)=\Ga(f)$
and 
%The middle term is given as
\[
\zeta^{er}(t)=\prod_{\De\in \mathcal M_0}(1-t^{d(P_\De}))^{(-1)^{n-1} \mu_{\De  }}
\]
where $\mu_{\De}$ is the sum of Milnor numbers $\mu(g_\De;q)$ for all  $q\in S(\De)$.
 $\zeta_\De(t)$ is the product of the  zeta function of $\hat f$ at $q\in S(\De)$. 
 
 If $f$ is almost Newton non-degenerate (so it  satisfies $(A2)$), $\zeta_\De(t)$ can be combinatorially computed by Varchenko formula.
\end{Theorem}
\begin{Remark}In \cite{Almost}, we have assumed the condition $ (A_2)$ and  $\hat f$ is pseudo-convenient at $q$ but these assumptions are not necessary. If $(A_2)$ condition is not satisfied,  the assertion is still true but to compute the zeta function $\zeta_\De(t)$,  we need an explicit resolution $\omega: Y\to X$ of $\hat f$ and then use the formula of  A'Campo instead of Varchenko's formula.
\end{Remark}

%%%%
\subsection{Zeta multiplicity}
By A'Campo formula, the zeta function $\zeta_f(t)$ of a germ of analytic function $f$ is written as
$\prod_{j=1}^s(1-t^{d_j})^{\nu_j}$ with mutually distinct $d_1,\dots, d_s$ and non-zero integers $\nu_1,\dots,\nu_s$. Thus we can write
${\rm{div}}(\zeta)=\sum_{i=1}^s\nu_i\La_{d_i}$.
 {\em  The zeta multiplicity of $f$} is defined  as $d_{min}:=\min\,\{d_i\,|\, i=1,\dots,s\}$ and we denote it as $m_\zeta(f)$. Suppose $d_{min}=d_{\iota},\,
 1\le \exists\iota\le s$. We call the factor $(1-t^{d_{\iota}})^{\nu_{\iota}}$ {\em the zeta multiplicity factor}.
 In general, $m_\zeta(f)\ge m(f)$ where $m(f)$ is the multiplicity of $f$, the lowest degree of the Taylor expansion of $f$ at $\bf 0$.
% If $f$ is a convenient non-degenerate function, $m_\zeta(f)=m(f)$ by Varchenko formula.
This follows from the following observation.
\begin{Proposition} \label{multiplicity estimation}
Assume that $\hat\pi: X\to  \mathcal U_0$ is a good resolution of an analytic function $f(\mathbf z)$  of multiplicity $m$
and put $\hat\pi\inv(V)=\tilde V\cup_{i=1}^s E_i$ where $\tilde V$ is the strict transform of $V=f\inv(0)$. Let $m_i$ be the multiplicity of $\hat \pi^*f$ along $E_i$. Then $m_i\ge m$.
\end{Proposition}
\subsection{L\^e-Ramunujam result for zeta-functions}
Consider a piecewise analytic family $f_s(\mathbf z),\,0\le s\le 1$ of functions with isolated singularity at the origin and suppose that the   Milnor number $\mu(f_s)$ of $f_s$ at $\mathbf 0\in \mathbb C^n$ is constant for $s$. ( $f_t$ can be piecewise $C^\infty$.)
Then 
\begin{Lemma}\label{zeta-const}
\label{zeta-constant}
The zeta function $\zeta_{f_s}(t)$ of $f_s$ is  independent of  $s$ and 
coincides with $\zeta_{f_0}(t)$.
\end{Lemma}
\begin{proof}For $n\ne 3$, the assertion follows from the result of  L\^e-Ramanujam \cite{Le-Ramanujam}.
For $n=3$, we consider the family $g_s(x,y,z,w)=f_s(x,y,z)+w^m$.
Consider the reduced zeta function
$\tilde\zeta(t):=\zeta(t)(1-t)$. For an isolated singularity case, $(-1)^n\tilde \zeta(t)$ is equal to the divisor of the characteristic polynomial of the monodromy automorphism $h_*:H_{n-1}(F)\to H_{n-1}(F)$ where $F$ is the Milnor fiber.
For a fixed $s$,  assume that
$\rm{div}(\tilde\zeta_{f_s})=\sum_{j=1}^{\ell} \mu_j\La_{e_j}$. 
By Join theorem (\cite{Thom-Seb,Sakamoto1}), we have the equality
$
\rm{div}(\tilde{\zeta}_{g_s})=\rm{div}(\tilde\zeta_{f_s})(-\La_m+1)$. Taking $m$ to be mutually prime for each $e_j$,
we have 
\[
\rm{div}(\tilde{\zeta}_{g_s})=-\sum_{j=1}^{k} \mu_j \La_{e_j m}+\sum_{j=1}^{k } \mu_j  \La_{e_j } \tag{J}
\]
Note that this divisor does not depend on the parameter $s$ by L\^e-Ramanujam (\cite{Le-Ramanujam}).
Assume that ${\rm{div}}(\tilde \zeta_{f_0})=\sum_{j=1}^{k_0} \nu_j d_j$.
%and ${\rm div}(\tilde \zeta_{f_1})=\sum_{i=1}^{k_1}\nu_i'\La_{d_i'}$. 
 We assume $d_1>d_2>\dots>d_{k_0}$ and $e_1>\dots>e_{k}$ and $m>1$. % is sufficiently large.
By the above equality, we get the equality
\[
-\sum_{j=1}^{k_0} \nu_j \La_{d_j m}+\sum_{j=1}^{k_0} \nu_j \La_{d_j }=-\sum_{j=1}^{k} \mu_j\La_{e_jm}+\sum_{j=1}^{k} \mu_j  \La_{e_j}
\]
for any $m$ which is coprime to any $\{d_1,\dots, d_{k_0},e_1,\dots, e_{k_1}\}$. We see that
$d_1=e_1,\,\nu_1=\mu_1$. By an inductive argument, 
we conclude that $k_0=k$ and $d_j=e_j,\,\nu_j=\mu_j$ for $j=1,\dots, k$.
\end{proof}
%%%%%%
%\end{enumerate}
%%%%%%
\section{ $\mu$-Zariski pairs}
\subsection{Zariski pairs and weak Zariski pairs}
A pair of  projective curves $\{C,C'\}$  in $ \mathbb P^2$ is called {\em a  Zariski pair} if 
they have the same degree and there is a bijective correspondence  
%between the singular points 
$\phi:S(C)\to S(C')$
where $S(C)$ and $S(C')$ are the sets of the singular points of $C$ and $C'$ respectively
and the local  topological type of the singularities of $(C,p)$ and $(C',\phi(p))$ is the same for any $p\in S(C)$  and $\phi$ extends to a homeomorphism $\tilde \phi: N(C)\to N(C')$ of a tubular neighborhood $N(C)$ of $C$ to a tubular neighborhood $U(C')$ of $C'$ but this does not extend to a homeomorphism
of the ambient spaces
$(\mathbb P^2,C)$ and $(\mathbb P^2,C')$ for any $\tilde \phi$.
% which extends  $\tilde\phi$. 

 We say $\{C,C'\}$ is {\em a weak Zariski pair} if 
 they have the same degree and there is a bijective correspondence $\phi: S(C)\to S(C')$ of the singular points of $C$ and $C'$ and the local  topological type of the singularities of $(C,p)$ and $(C',\phi(p))$ is the same for any $p\in S(C)$.
However there does not exist any homeomorphism 
 $\tilde \phi:N(C)\to N(C')$ of their tubular neighborhoods which extends $\phi$.
 This implies in particular that  the pairs $(\mathbb P^2,C)$ and $(\mathbb P^2,C')$ are not homeomorphic.
 \subsection{ Zariski pairs of hypersurfaces}
Assume that we have a pair of hypersurfaces $V(f)=\{f(\mathbf z)=0\}$ and $V(g)=\{g(\mathbf z)=0\}$ with isolated singularity at the origin. We say $\{V(f), V(g)\}$ is a {\em $\mu$-Zariski pair of hypersurface} (respectively {\em $\mu^*$-Zariski pair of hypersurfaces}) if they have  a same Milnor number $\mu$ (respectively a same $\mu^*$-invariant) and a same zeta function  of the Milnor fibrations
but they belong to different connected components of $\mu$-constant strata (resp. of $\mu^*$-constant strata). For the definition of $\mu$-constant strata  and $\mu^*$-constant strata, see \cite{EO}. They are defined as  semi-algebraic sets.

There is a canonical way to produce  possible $\mu$-Zariski pairs of surfaces ($n=3$).
Consider a  Zariski pair (respectively a weak Zariski pair) of projective curves  $C, C'$ of degree $d$ defined by  convenient homogeneous polynomials $f_d(x,y,z)$ and $g_d(x,y,z)$. We assume that the singular points of $C$ and $C'$ are Newton non-degenerate with respect to some local coordinates.
We   assume that $f$ and $g$ are non-degenerate on any face $\De$
of their Newton boundary with $\dim\De\le 1$.
Consider the affine  surfaces  defined by 
$f(x,y,z)=f_d(x,y,z)+z^{d+m}$ and $g(x,y,z)=g_d(x,y,z)+z^{d+m}$. 
Then $f$ and $g$ are almost non-degenerate functions with isolated singularities at the origin and their zeta functions  and Milnor numbers are same. We call $\{f, g\}$ {\em a Zariski pair } (resp.{\em a weak Zariski pair}) of  surfaces (\cite{Almost}). In \cite{Almost}, we studied a Zariski pair of  surface with $m=1$ whose links are diffeomorphic.
In our  paper  \cite{EO} in preparation, we have shown that the pair $\{f, g\}$ defined as above starting  from a Zariski pair $\{f_d, g_d\}$ 
of projective curves is a $\mu^*$-Zariski pair of  surfaces. Hereafter in this paper we consider  mainly a pair of  surfaces constructed from a weak Zariski pairs of curves.
%%%%%%%%%%%
\subsubsection{Examples of weak Zariski pairs of surfaces}

We consider   two weak Zariski pairs of quartics in $\mathbb P^2$.
Recall that a pair of projective curves $\{C_1,C_2\}$ of the same degree is a {\em weak Zariski pair} if there is a bijection  $\psi:S(C_1)\to S(C_2)$ of the respective  singular points 
so that the topological singularity type $(C_1,q)$ and $(C_2,\psi(q))$ are equivalent for any $q\in S(C_1)$ but this homeomorphism does not extend to a homeomorphism of  any tubular neighborhoods $N(C_1)$ and $N(C_2)$ of $C_1$ and $C_2$ respectively.
% the pair $(\mathbb P^2,C_1)$ and $(\mathbb P^2,C_2)$ are not homeomorphic.
Examples of weak Zariski pairs which we consider in this paper are:
\begin{enumerate}
\item[( a1)] $(Q_1,Q_2)$ where $Q_1$ is an irreducible quartic with 3 nodes i.e., 3 $A_1$ singularities. $Q_2$ is union of a smooth cubic and a generic line.
%Both of $Q_1$ and $Q_2$ has 3 $A_1$ singularities.
\item[(a2)] $(Q_3,Q_4)$ where $Q_3$ is a union of a cubic with one node and a generic line
and   $Q_4$ is a union of two conics which intersects transversely at 4 points.
%They have 4 $A_1$ singularities.
\end{enumerate}
Note that $Q_1,Q_2$ has $3 A_1$ singularities and $Q_3,Q_4$ have  $4 A_1$ singularities.
%Example of $Q_1$ is given as 
%$Q_1:\, y^{4}+y^{2} \left(-x^{2}+x -2\right)-x^{4}+x^{3}+x^{2}$.
More explicitly, as $Q_1$ and $Q_2$, we can take (see \cite{Fermat}):
\[\begin{split}
Q_1:\quad &q_1(x,y,z)=0,\\
&q_1=(x^4+y^4+z^4)-124
(xyz^2+xy^2z+x^2yz)
 +6(x^2z^2+y^2z^2+x^2y^2)\\
&-4(x^3y+xy^3+x^3z+xz^3+y^3z+yz^3)\\
Q_2: \quad &q_2=0,\\
& q_2(x,y)=(x^3+y^3+z^3)(x+ay+bz),\,\, a,b\in \mathbb C^{*}:\text{generic}
\end{split}
\]
and 
 $Q_3=C_3\cup L$ and $C_3\pitchfork L$ where $C_3$ is a cubic with  one node.
For example, as  $C_3$ with one node at $(1,1,1)$ we can take 
\[
C_3:\, c^{ (1)}_3(x,y,z)=-(x+y+z)^3+27xyz
\]
and adding a generic line component, we get such a quartic $Q_3$. For example, 
we take  $q_3(x,y,z)=c^{ (1)}_3(x,y,z)\times(x+2y+3z)=0$. It has 4 nodes, three of which  come from the intersection of $C_3$ and the line component. As $q_4$ we can take for example,
$q_4=(x^2+y^2+z^2)(x^2+2y^2+3z^2)$.
As affine hypersurface,  each quartic $q_i(x,y,z)=0$ has three (or four) singular  lines through the origin
for $i=1,2$ (respectively for $i=3,4$). In the following, the precise forms of $q_1,\dots, q_4$ are not important. They make no problem for the discussion below.
%%%%
\subsubsection{Isolation of the singularities}\label{weakZariski}
We consider the following  polynomials which is associated with $q_i,\,i=1,\dots,4$:
\[
f_i(x,y,z)=q_i(x,y,z)+z^{4+m},\quad i=1,\dots, 4
\]
where $m$ is a fixed positive integer. 
$f_1,\dots, f_4$ are   almost non-degenerate functions
and their Milnor numbers  are given by
$27+3m$ for $f_1,f_2$ and  $27+4m$ for $f_3,f_4$. 
Consider the corresponding hypersurface $V_i=\{f_i(x,y,z)=0\},\,i=1,\dots,4$.
For the calculation of the zeta function,  we follow the procedure of Theorem \ref{main1}. We first take an ordinary blowing up which is the simplest toric modification
$\hat \pi:X\to   \mathbb C^n$ with one   positive weight vector $P={}^t(1,1,1)$.
%Here $\mathcal U_0$ is a small neighborhood of the origin.
Take the toric chart
 $\Cone(e_1,e_2,P)$.  The exceptional divisor $E(P)=\hat E(P)\cap \tilde V_i$ contains 3   nodes $\rho_1,\rho_2,\rho_3$ for $V_1, V_2$ and  4 nodes
 $\rho_1,\dots, \rho_4$ for $V_3, V_4$.
 Taking the toric coordinate $(u_1,u_2,u_3)$ ,
 we have 
 \[
\hat f_i(\mathbf u):= \hat\pi^*f_i(u_1,u_2,u_3)=u_3^4\left\{
 q_{i}(u_1,u_2,1)+u_3^{m}.
 \right\}
 \] 
 Recall $(x, y, z)=(u_1u_3, u_2u_3, u_3)$. 
 %where $m$ is a fixed positive integer.
 Consider the behavior at a node $ \rho_\al$.
 Taking   admissible coordinates $(v_1,v_2,v_3)$ with $v_3=u_3$  in  a neighborhood of $\rho_\al$
 so that
   \begin{eqnarray}\label{local-equation}
  \hat q_{i,\al}(\mathbf v)&=&v_3^4(v_1^2+v_2^2),\\
 \hat f_{i,\al}(\mathbf v)&=&v_3^4(v_1^2+v_2^2+v_3^m).
 \end{eqnarray}
Thus the zeta function $\zeta_{\hat f_{i,\al}}(t)$  at $\rho_\al$ is given as
$\zeta_{\hat f_{i,\al}}(t)=(1-t^{m+4})^{-1}$ by Lemma \ref{Va}. The geometry at other singular point $\rho_\be$ is exactly  same with that of $\rho_\al$. Thus using Theorem \ref{main1}, combining the zeta functions at $\rho_\al$, we get
\begin{eqnarray}\label{zeta12}
\zeta_{f_{i} }(t)&=(1-t^4)^{-4} (1-t^{4+m})^{-3},\,i=1,2 \\%\notag \\
\zeta_{f_{i}}(t)&=(1-t^4)^{-3}(1-t^{4+m})^{-4},\,i=3,4.\label{zeta34} %\notag
\end{eqnarray}
Note that the generic plane sections of $f_1,\dots, f_4$ have non-degenerate convenient  degree $4$ components. Therefore
the $\mu^*$-invariant of $f_i$ is given as
\begin{eqnarray}\label{mustar12}
\mu^*(f_i)=
\begin{cases}(27+3m,9,3),&\quad i=1, 2\\
(27+4m,9,3),&\quad i=3, 4.
\end{cases}
\end{eqnarray}
We will show that $\{f_1,f_2\}$ and $\{f_3,f_4\}$ are $\mu$-Zariski pairs in the following sections.
\subsection{Main theorem}We consider the isolation pairs $\{f_1,f_2\}$ and $\{f_3,f_4\}$ of
the weak Zariski pairs 
$\{q_1,q_2\}$ and $\{q_3,q_4\}$ introduced  in \S \ref{weakZariski}.
\subsubsection{ Non-existence of $\mu^*$-constant path}
First we assert
%%%%%
\begin{Lemma}\label{mu*-zariski}
There are no piecewise analytic $\mu^*$-constant  path from
 $f_1$ to $f_2$ (respectively  from $f_3$ to $f_4$). \end{Lemma}
%\begin{comment}
\begin{proof} 
 The assertion is proved in \cite{EO} for a pair constructed from Zariski pair and the proof for our case is similar. We   give a brief proof for the reader's convenience.
We prove the assertion simultaneously for two pairs. %For simplity we write $h_s=h$ hereafter.
Suppose we have a  piecewise analytic family of functions $h_s(x,y,z),\,0\le s\le 1$ so that $h_0=f_1$ and $h_1=f_1$ (respectively $h_0=f_3$ and $h_1=f_3$) and $\mu^*$ invariants of $h_s$ are constant.
As it is $\mu^*$-constant family, the multiplicity of $h_t$ is 4 for any $s$. Let $h_s=h_{s4}+h_{s5}+\dots$ be the graduation by the degree.
%First $h_{s4}=0$ has  only isolated singularities in $\mathbb P^2$. 
If $h_{s_0,4}=0$  has an non-isolated singularity in $\mathbb P^2$ for some $s_0$,  the generic plane section $h_{s_0}\cap L$ has Milnor number greater than 9 as the tangent cone
is a quartic with  singularities. Here $L$ is a generic plane through the origin. This contradicts to the $\mu^*$-constancy. Thus $h_{s4}=0$ has only isolated singularities.
Secondly the family of quartics $h_{s4}(x,y,z)=0$  has the same total Milnor number. In fact, the zeta function of $h_s$, $\zeta_{h_s}(t)$ is constant for $s$ by Lemma \ref{zeta-const} and also its zeta multiplicity factor  is also constant.  This  is given by
$(1-t^4)^{-7+\mu_{tot}(s)}$. Here $\mu_{tot}(s)$ is the total Milnor number of the projective curve
$D_s:=\{h_{s4}=0\}\subset \mathbb P^2$ which is the sum of Milnor numbers at the singular points. Thus the total Milnor numbers of $h_t$ is constant.

We use the following well-known property of the family of curves.
\nl
{\em Bifurcation of singularities.}
Consider  a  continuous family of analytic function $f_t(x,y)$ of plane curves with isolated singularity at the origin.
%The Milnor number is equal to the local mapping degree of the Jacobian mapping.
% and therefor a family of isolated singularities at the origin. 
Then there exists a positive number $\eps>$ so that for any $t\le \eps$, $\mu(f_t)\le \mu(f_0)$ for $|t|\le \eps$. It is also well-known that if the singularity of $f_0$ at the origin  bifurcate into some singularities for $f_t$, the sum of Milnor numbers on the same fiber $f_t=0$ is   less than $\mu(f_0)$ by Lazzeri \cite{Lazzeri}. See also \cite{Le2}.
% as the Lefschetz number of the monodromy is $0$ by the result of A'Campo \cite{ACampo2}.

 Thus combining the above  two observation, the singularities of $\{h_{s4}=0\}$
has  $3 A_1$ (respectively   $4 A_1$)  for any $s$ if $h_0=f_1$ (resp. if $h_0=f_3$). This implies that  the pair $(\mathbb P^2,D_s)$ is topologically isomorphic  to $(\mathbb P^2,D_0)$ by a result of L\^e \cite{Le}.
However this is a contradiction as $\{Q_1,Q_2\}$ (resp.$ \{Q_3,Q_4\}$) is a weak Zariski pair and the pair $(\mathbb P^2,Q_i),i=1,2$ (resp. the pair $(\mathbb P^2,Q_j),j=3,4$) are not homeomorphic.  Thus
 there are no such family of quartics  from $Q_1$ to $Q_2$ (resp. from $Q_3$ to $Q_4$).  This proves Lemma \ref{mu*-zariski}.
\end{proof}
%\end{comment}
\subsubsection{$\mu$-Zariski pair}
Now we state a stronger  result. 
\begin{Theorem}\label{mu-zariski}
The pair $\{f_1,f_2\}$ and $\{f_3,f_4\}$ are $\mu$-Zariski pairs of hypersurfaces. Namely they belong to  different connected components of $\mu$-constant strata.
\end{Theorem}
\section{Proof of Theorem \ref{mu-zariski}}
The proof occupies the rest of this section. 
 Assume that we have a $\mu$-constant  piecewise analytic family $h_s(x,y,z),\,s\le s\le 1$ so that $h_0=f_1,h_1=f_2$ (respectively $h_0=f_3, h_1=f_4$). We take an arbitrary $0<s<1$. 
 We will show that zeta function can not be the same as any  of $f_1$ or $f_3$ if the multiplicity of $h_s$ is smaller than 4. This part takes the most part of the proof.
 By Lemma \ref{zeta-const}, the zeta-function of $h_s$ is the same as that of $f_1$ or $f_3$.
We prove the assertion by contradiction.

 The argument is to show that the zeta multiplicity factor of $h_s$ can not be  as $(1-t^4)^{-4}$ for $f_1$ or  $(1-t^4)^{-3}$ for $f_3$ if multiplicity of $h_s$ is less than 4. 
 (There is one exceptional case with $\zeta$ multiplicity factor is 
 $(1-t^4)^{-3}$  and the multiplicity is $3$. See Lemma \ref{exceptional}.)
 If  the multiplicity of $h_s$ is 4,   the singularities of $h_{s4}=0$ must be $3A_1$ (resp. 4 $A_1$).
 We first show that the multiplicity of $h_s$ can not be 2 or 3
 in \S \ref{Case1} and \S\ref{Case2}.

\vspace{.3cm}
\noindent 
\subsection{Case 1} \label{Case1} {\em The multiplicity of $h_s$ is  2.} Assume $h_s$ has the multiplicity 2 for some $s$.
Fixing $s$ and apply the generalized Morse Lemma (see for example \cite{Arnold}). Choosing a suitable analytic coordinate $(w_1,w_2,w_3)$,
we can write (a) $h_s(\mathbf w)=w_1^2+w_2^2+w_3^\nu,\,\nu\ge 3$ for corank 1 or
(b) $h_s(\mathbf w)=w_1^2+j(\mathbf w)$ for corank 2 where the multiplicity of $j$ is greater than 2. We show that this is impossible, under the assumption that the zeta function is given as (\ref{zeta12}) or (\ref{zeta34}).
 
For  the case (a), it is clearly impossible, as $\text{div}(\tilde\zeta_{h_s})=\La_\nu-1$. 
  Assume the case (b). Let $\Xi_j$ be the  divisor of the reduced zeta function of $j(\mathbf w)$.  By the join theorem (\cite{Thom-Seb, Sakamoto1}),
we need to have
\begin{eqnarray}
{\rm div}(\tilde\zeta_{h_s})&=&(-\La_2+1)\Xi_j,\\
(-\La_2+1)\Xi_j&=&{\rm div}(\tilde\zeta_{h_0})=
\begin{cases}
-4\La_4-3\La_{4+m}+1,\,\,&\text{for}\,f_1,f_2\\
-3\La_4-4\La_{4+m}+1,\,\, &\text{for}\,f_3,f_4.
\end{cases}
\end{eqnarray}
%where $\Xi_j$ is the divisor of $\tilde \zeta_j(t)$. 
Put $\Xi_j=\sum_{i=1}^s \nu_i\La_{d_i}$ with $d_1<d_2<\dots<d_s$.
First, to obtain $1$ in $(-\La_2+1)\Xi_j$, we must have $d_1=1$ and $\nu_1=1$.
%First, assume $m$ is odd.  To have the divisor $\La_{4+m}$, we need to have $d_j=4+m$ for some $j$.  Then ${\rm div}(\tilde\zeta_{h_s})$ get $\nu_j\La_{2(4+m)}$ with non-zero coefficient.
 If $d_2>2$,
$(-\La_2+1)\Xi_j$ gets  $-\La_2$ in this summation. This is a contradiction to the above equality. So we need to have $d_2=2$ and $\nu_2=-1$. 
This implies by Proposition \ref{multiplicity estimation}, the multiplicity of ${j}$ is 2 which is also a contradiction to the assumption.

%\vspace{.3cm}\noindent
\subsection{Case 2}\label{Case2}{\em Multiplicity of $h_s$ is 3.}
Now we show that the multiplicity of $h_s$ can not be $3$.  
Assume that $h_s$ has multiplicity 3 for some $s$ and let $h_s=h_{s3}+h_{s4}+\dots$ be the graduation by the degree. We consider the cubic curve
$C=\{h_{s3}=0\}\subset \mathbb P^2$.  In the following, $s$ is fixed as above.
\subsubsection{Strategy of the argument}
Our argument proceeds as follows.
First we take a suitable coordinates, say $(x, y, z)$, and consider the Newton boundary  $\Ga(h_s)$ of $h_s$ with respect to this coordinates. As $h_s=0$ 
has an isolated singularity at the origin, we may assume that $h_s$ has a  convenient  Newton boundary by adding monomials $x^N, y^N, z^N$ where $N$ is a sufficiently large integer. Then we consider the dual Newton diagram $\Ga^*(h_s)$ and take an admissible regular simplicial subdivision  $\Si^*$ and consider the associated toric modification
$\hat \pi:X\to \mathbb C^3$. By the convenience,  we may assume that the vertices of $\Si^*$
are   positive except the canonical ones $\{e_1,\dots, e_n\}$ and $\hat \pi$ is a small toric modification. Let $\hat E(P),\,P\in \mathcal V^+$ be the compact exceptional divisors of $\hat \pi$. 
%Let $P_j$ be the primitive   positive weight vectors such that $\hat E(P_j)=\hat E_j$. 
The multiplicity of $\hat h_s:=\hat \pi^*h_s$ along $\hat E(P)$ is $d(P,h_s)$. Let $\tilde V_s$ be the strict transform of $V(h_s)$ into $X$.
If $\De(P)$ is a degenerate face of $\Ga(h_s)$, $\tilde V_s$ can have singularities on $ E(P)$. 
%If $h_{P_j}$ is degenerate non monomial factor, $\tilde V$ has a non-isolated  singularity on $\hat E_j$. 
To get a good resolution of $h_s$, we need  further blowing ups  over   singular points of $\tilde V_s$ (\cite{Hironaka}) and let
$\omega:Y\to X$ is the composition of these blowing ups so that the composition 
\[\Pi=\hat\pi\circ\omega: Y\mapright{\omega} X\mapright{\hat\pi}\mathbb C^3\]
 is a good resolution of $h_s$ and let $D_1,\dots, D_\ell$ be the exceptional divisors of $\omega$
 and let $m_j$ be the multiplicity of $\Pi^*\hat h_s$ along $D_j$. 
 Note that  
 $m_j\ge 5$ by Proposition \ref{multiplicity estimation}, if the multiplicity of the exceptional divisor $\hat E(P)$ of the first modification $\hat\pi:X\to\mathbb C^3$ with
 $\omega(D_j)\subset \hat E(P)$ is at least 4,  which implies the multiplicity of $\hat\pi^*f$ is greater than or equal to $5$ at a singular point of $\tilde V$.
 % as the exceptional divisors $\hat E(P)$ on which $E(P)$ has a singular point has multiplicity greater or equal to 4 
 %by Proposition \ref{multiplicity estimation}. 
 Let $\tilde V_Y$ be the strict transform of $\tilde V$ into $Y$ and 
 $D_j'= D_j\setminus \left(\tilde V_Y\cup_{P\in \mathcal V^+}\hat E(P)_Y\cup_{k\ne j}D_j\right)$.
 We may assume that exceptional divisors are all compact so that its image  of the exceptional divisors by $\Pi$  are   over the origin.   Then the exceptional divisors of $\Pi=\hat\pi\circ\omega$
are $\{\hat E(P) _Y,\, P\in \mathcal V^+\}\cup\{D_1,\dots, D_\ell\}$. 
Here $\hat E(P)_Y$ is the pull back of $\hat E(P)\subset X$ to $Y$.
The contribution of the divisor $\hat E(P)_Y$ in the A'Campo formula is $(1-t^{d(P)})^{-\chi(\hat E(P)_Y')}$
where  $\hat E(P)_Y'=\hat E(P)_Y\setminus  \left({\tilde V_Y}\cup_{Q\ne P}\hat {E}(Q)_Y\cup _{i=1}^kD_k\right)$.
Let $\hat E(P)'=\hat E(P)\setminus \left(\tilde V\cup_{Q\ne P}\hat E(Q)\subset \hat E(Q)\right)$. 
As $E(P)'$ is smooth and it  does not contain any point of the center of the second blowing-up $\omega$,
we have a canonical diffeomorphism $\omega:\hat  E(P)_Y'\cong\hat E(P)'$.
 Thus
\begin{Proposition}
 %The contribution of ${\hat  E}(P)_Y'\subset Y$ to the zeta function $\zeta_{h_s}(t)$ is the same as that of $\hat E(P)'\subset X$.
 We have the equality $\chi({\hat  E}(P)_Y')=\chi({\hat  E}(P)')$.
 
 \end{Proposition}

 %By the argument of 
 
 Now combining  A'Campo formula and the argument of Varchenko formula and Proposition \ref{vanishing}, we have 
 \begin{Lemma}\label{modified-Varchenko}
  The zeta function of $f$ is given as 
 \[\prod_{P\in \mathcal V^+}(1-t^{d(P)})^{-\chi(\hat E(P)')}\times \prod_{j}(1-t^{m_j})^{-\chi(D_j')}.
 \]
 %and $m_j\ge 5$.
 The first factor can be written using toric stratification as 
 \[\begin{split}
 \prod_{P\in \mathcal V^+}(1-t^{d(P)})^{-\chi(\hat E(P)')}&=\prod_{I}\zeta_I(t)\\
 \text{where}\,\qquad\zeta_I(t)&=\prod_{Q\in \mathcal P_I}(1-t^{d(Q)})^{-\chi(\hat E(Q)')}
 \end{split}
 \]
 \end{Lemma} 
 The set $\mathcal P_I$ is the set of  weight vectors which correspond to the maximal dimensional faces of $\Ga(f^I)$. If $f^I_P$ is a degenerate face, $\chi(\hat E(P))$ can not be expressed combinatorially as in the formula (\ref{Va2}).
 The following Lemma is useful to prove Theorem \ref{mu-zariski}.
 Using Lemma \ref{vanishing}, we have:
 %First we prepare a Lemma.
 \begin{Lemma}\label{essential vertex}
  Let $P$ be a   positive vertex of $\Si^*$.
 If $\hat E(P)'$ has non-zero Euler characteristic, there are three possibilities.
\begin{enumerate}
\item $\dim\, \De(P)=2$, or
\item $\dim\,\De(P)=1$ and $P$ is adjacent to one of  $e_1, e_2, e_3$, or
\item $\dim\, \De(P)= 0$ and $P$ is adjacent to two of $e_1, e_2, e_3$.
\end{enumerate}
 \end{Lemma}
 \begin{proof}  Assume that $\dim\,\De(P)=1$. Take a toric coordinate chart  $\si=\Cone(P,P_2,P_3)$. If $P$ is not adjacent to any of $e_1,e_2,e_3$,
 $\hat E(P)'=\hat E(P;\si)^*\setminus E(P;\si)^*$ for any toric chart $\si=\Cone (P,P_2,P_3)$
 where 
 \[\begin{split}
 \hat E(P;\si)^*&:=\{\mathbf u_\si\in \mathbb C_\si^3\,|\, u_{\si 1}=0,\, u_{\si 2},u_{\si 3}\ne 0\}\\
 E(P;\si)^*&:=\{(0,u_{\si 2},u_{\si 3})\,|\, g(u_{\si 2},u_{\si 3})=0\}
 \end{split}
 \]
 and   $g$ is the defining polynomial of $E(P)$ in $\hat E(P)=\{u_{\si 1}=0\}$. By the assumption, the Newton polygon of $g$ is 1-dimensional. Thus by Lemma \ref{vanishing}, $\chi(\hat E(P)')=-\chi(E(P;\si)^*)=0$. If $P$ is adjacent to $e_1$,  $\De(P)\subset \mathbb R^{\{2,3\}}$.
 The proof of assertion (3) is similar.  In this case,  $\hat E(\si)^*$ is a point (=0-dimensional torus) for $\si$ which is generated by $P$ and two of $e_1,\dots, e_3$. For example, if $e_1,e_2$ is adjacent to $P$, $h_{sP}(x,y,z)=c\,z^a$ for some $a>0$ and $c\in \mathbb C^*$ and $P$ take the form $P={}^t(a,b,1)$,\, $a,b>0$.
 In this case, this vertex contributes the zeta function by $(1-t^a)^{-1}$.
 \end{proof}
% To have zeta multiplicity factor $(1-t^4)^{-4+\eps},\,\eps=0,1$, we need a face of $\Ga(h_s)$ which can have weight 4, as 
% the multiplicity of the divisors $D_1,\dots, D_\ell$ by the blowing up $\omega$ are bigger than or equal to 5.
\subsubsection{Cancellation of Case 2}
% Multiplicity of $h_s$ can not be 3}
Now we are ready to show the impossibility of the multiplicity $m(h_s)=3$.
We divide the situation by the geometry of the cubic curve $C_3:=\{h_{s3}=0\}\subset \mathbb P^2$.
%\nl\noindent
For simplicity, we write hereafter $h:=h_s$, $h_3=h_{s3}$  etc.

\noindent
We divide the case 2 in  three subcases.
\indent
\begin{enumerate}
\item [ 2-1] $C_3:\,h_3=0$ a union of 3 lines. 
\item[  2-2] $h_3=0$ is a union of conic and a line.
\item [  2-3]$h_3=0$ is an irreducible cubic.
\end{enumerate}
First we consider Case 2-1.\nl
{\bf Case 2-1}. $C_3:\,h_3=0$ a union of 3 lines. 

%Observe that $h_{4}\ne 0$. Otherwise, there are no possible factor of 
%$(1-t^4)$.
 We further divide this case into four cases depending the geometry of the lines $h_3=0$: \newline
(a) $C_3$ is a union of three lines which  are generic in $\mathbb P^2$,  or\nl
(b) $C_3$  is a union of  three lines  which are intersecting at one point in $\mathbb P^2$.
\nl
(c) $C_3$ is union of two lines where one line has multiplicity 2.
\nl
(d) $C_3$ is one  line with multiplicity 3.
\nl\indent
 { Subcase 2-1-a.} We first consider the case that  $h_3=0$ is three generic lines. Taking a new linear coordinate if necessary
and putting this coordinates as $(x,y,z)$, we may assume 
% coordinate $(w_1,w_2,w_3)$ so 
 that $h_3(x,y,z)=xyz$.
 Put $A=(1,1,1)\in \Ga(h)$.
Let $P={}^t(a_1,a_2,a_3)$ be a   positive vertex of $\Si^*$ with $d(P)\le 4$. If $d(P)=3$, then $P={}^t(1,1,1)$ and 
 $\De(P)=\{A\}$.
This vertex does not give any factor in zeta function and  we do not need to  consider this case. Suppose $d(P)=4$. Then  $\De(P)$ contains only degree 4 monomials and possibly $xyz$. 
Here we used the trivial inequality
$\deg\, x^ay^bz^c\ge a+b+c$. 
%For simplicity of notation, we assume that $\mathbf w=(x,y,z)$ from the beginning so that $h_3=xyz$.
If $\hat E(P)'$ has non-zero Euler characteristic, the possibility is  (a) $\dim\, \De(P)=2$, or (b)
$\dim\,\De(P)\ge1$ and $P$ is adjacent to one of $e_1, e_2, e_3$,or
(c) $\dim\, \De(P)\ge 0$ and $P$ is adjacent to two of $e_1, e_2, e_3$.
For (a) or (b), the possible weights are ${}^t(2,1,1), {}^t(1,2,1), {}^t(1,1,2)$.
Thus we may assume, for example, that $P={}^t(1,1,2)$ and  any degree $4$ monomial except $xyz$ must be a monomial $x^i y^{4-i}$ of degree $4$ in 
$x,y$. Thus we assume that 
 $I=\{1,2\}$ and suppose $h^I(x, y)$ 
  have  1-dimensional support.
We first assume that $\si=\Cone\,(P,e_2,e_3)$   is a simplex in $\Si^*$, assuming $x^4$ is in $h_4^I$. (Note $\De(P)\cap\De(E_2)\cap\De(e_3)=\{(4,0,0)\}$ and we can consider $\si\in \Si^*$.)
As a unimodular matrix, $\si$ takes the the form:
\[
\si=\left(
\begin{matrix}
1&0&0\\1&1&0\\2&0&1
\end{matrix}
\right).
\]
Then $\hat\pi_\si(\mathbf u_\si)=
(u_{\si 1},u_{\si 1}u_{\si 2},u_{\si 1}^2u_{\si 3})$ and 
\[\begin{split}
&\hat h(\mathbf u_\si)\equiv\hat{h_P}(\mathbf u_\si)
\,\,\modulo (u_{\si 1}^5),\\
&\hat h_P(\mathbf u_\si)=u_{\si 1}^4\left (h_4^{I}(1,u_{\si 2})+u_{\si 2}u_{\si 3}\right). %\,\,\modulo\,(u_{\si 1}^5).
\end{split}
\]
Let $\nu_I$ be the number of non-zero distinct roots of $h^{I}(1,u_{\si 2})=0$
and let $\de_I$ be the number of monomials of  $\{x^4,y^4\}$ in $h^I$. Then
\[
E(P)^*=\{(u_{\si 2},u_{\si 3})\,|\, h_4^{I}(1,u_{\si 2})+u_{\si 2}u_{\si 3}=0,\, u_{\si 2}, u_{\si 3}\ne 0\}
\]
and it is easy to see that $E(P)^*$ is homeomorphic to $\mathbb C^*\setminus \{\nu_I\,\text{points}\}$ by the projection $(u_{\si 2},u_{\si 3})\mapsto u_{\si 2}$.
Thus $\chi(E(P)^*)=-\nu_I$. On the other hand,
$E(P)_I^*=\{u_{\si 2}\in \mathbb C^*|h_4^{I}(1,u_{\si 2})=0\}$ and $\chi(E(P)_I^*)=\nu_I$. Thus those two terms are cancelled out.
Thus the contribution of the stratum  $E(P)^*$ and $E(P)_I^*$ to the zeta multiplicity factor is 
$(1-t^4)^{\de_I}$ where $\de_I$ is 
%2, 1, 0 according to
 the number of monomials in  $\{x^4, y^4\}$ which are in $h_4^{I}$
 and $\de_I\in \{0,1,2\}$.
If $y^4$ appears in $h_4^I$ and $x^4$ does not appear in $h_4^I$, we do the same argument by $\si'=\Cone(P,e_1,e_3)$.
If $x^4$ and $y^4$ are not in $h_4^I$ and assume that $h_4^I(x,y)=y^\al j_{4-\al}(x,y),\,1\le\al\le 2$ where $j_{4-\al}$ is a polynomial of degree $4-\al$ with $j_{4-\al}(x,0):=c\ne 0$.
Take a vector $Q={}^t(a-1,a,b)$ with $a$ sufficiently large and $b\gg a$.
We can see $h_Q=c y^{\al} x^{4-\al}$ and $\De(P)\supset\De(e_3)\supset \De(Q)$. (Recall $\Cone\,(P,Q,R)$ is an admissible cone if $\De(P)\cap\De(Q)\cap \De(R)\ne \emptyset$.)
This means 
\[
\tau=\Cone(P,Q,e_3)\,\iff\,\tau=
\left(
\begin{matrix} 1&a-1&0\\ 1&a&0\\ 2&b&1
\end{matrix}
\right)
\]
is an admissible regular simplicial cone. We may assume that $\tau$ is a simplicial cone of  a regular simplicial cone subdivision $\Si^*$ of $\Ga^*(h)$. However this choice of $\tau$ and  an explicit construction of $\Si^*$ is not necessary and this particular choice of $\tau$ does not make any difference in the calculation of $\chi(\hat E(P)')$ which is clear from the following calculation. In fact, 
in this coordinate chart, $\hat \pi^* h$ is defined by
\[\begin{split}
\hat h(\mathbf u_\tau)&=u_{\tau 1}^4u_{\tau 2}^{4a-4+\al}  \left(j_{4-\al}(1,u_{\tau 2})+
u_{\tau 2}^{b-2a+3-\al}u_{\tau 3}\right )\,\modulo\,(u_{\tau 1}^5)
%g(u_{\tau 2},u_{\tau 3})&=j_{4-\al}(1,u_{\tau 2})+u_{\tau 2}^{b-2a+3-\al}u_{\tau 3}.
\end{split}
\]
and $E(P)$ is defined by
\[g(u_{\tau 2},u_{\tau 3})=j_{4-\al}(1,u_{\tau 2})+u_{\tau 2}^{b-2a+3-\al}u_{\tau 3}=0.
\]
Thus we see $E(P)^*$ is isomorphic to $\mathbb C^* \setminus \{\nu_I\,\text{points}\}$ (isomorphism is given by the projection
$(u_{\tau 2}, u_{\tau 3})\mapsto u_{\tau 2}$)
and $E(P)_I^*$ is $\nu_I$ points which are   roots of $j_{4-\al}(1,u_{\tau 2})=0$. Note that $\nu_I$ does not depend on the choice of $b\gg a\gg1$.
%%%
We do the same discussion for $J=\{2,3\}$ and $K=\{1,3\}$ and we conclude the zeta-multiplicity factor is given as 
$(1-t^4)^{-\de}$ where $\de$ is the number of monomials in $\{x^4, y^4, z^4\}$ in $h_{4}$.
Thus if $\de<3$, the zeta multiplicity factor is $(1-t^4)^{-\de}$ and it can not be same with that of $f_i,\, i=1,3$, $(1-t^4)^{-4+\eps},\,\eps=0, 1$.
The case $\de=3$ is a bit different, as the zeta-multiplicity factor coincides with that of $f_3,f_4$.
In  this case,  $h_4$ contains  three monomials $x^4, y^4, z^4$ and %we say that 
$h_{4}$ is convenient. We assert
\begin{Lemma}\label{exceptional}
Assume that $h_3=xyz$ and $h_{4}$ is convenient.
%has a full Newton polygon. 
Then the zeta function of $h$ is given as $(1-t^4)^{-3}$ and the Milnor number is $11$.
\end{Lemma}
Thus assuming this lemma, $h$ can not have the same zeta-function as $f_1$ or $f_3$
which are $(1-t^4)^{-4+\eps}(1-t^{4+m})^{-3-\eps}$, $\eps=0,1$.
\begin{proof}
We choose another linear coordinate $(x',y',z')$ so that 
$x=\ell_1:=(a_1x'+a_2y'+a_3z'),y=\ell_2:=(b_1x'+b_2y'+b_3z'), z=\ell_3:=(c_1x'+c_2y'+c_3z')$ where $a_i,b_i,c_i\ne 0, i=1,2,3$ are generic non-zero complex numbers and they satisfy
$h_{4}(a_i,b_i,c_i)\ne 0$ for $i=1,2,3$. 
 In this coordinate, we have $h_3=\ell_1\ell_2\ell_3$ and consider the homogeneous polynomial
$H_{3}(x',y',z'):=h_{3}(\ell_1,\ell_2,\ell_3)$.
The intersection points in $\mathbb P^2$ of three lines $x=0,y=0,z=0$ are 
%in $x,y,z$ coordinates, 
$\rho_1=(1,0,0)$, $\rho_2=(0,1,0)$ and $\rho_3=(0,0,1)$. In the new coordinates $(x',y',z')$, we have:
% put these intersection points
%$\rho_1,\rho_2,\rho_3$.  
\begin{Assertion} In the coordinates $(x',y',z')$, $\rho_i,\, i=1,2,3$ are not on the coordinate lines $ \{x'y'z'= 0\}$ and $H_3(x', y', z')$
is convenient.
% has a full Newton polygon.
\end{Assertion}
\begin{proof} By solving  explicitly respective linear equations
$\ell_1-1= \ell_2=\ell_3=0$, $\ell_1=\ell_2-1=\ell_3=0$ and $\ell_1=\ell_2=\ell_3-1=0$ in $x',y',z'$,   we can easily see that $\rho_1,\rho_2,\rho_3$ are outside of the lines $x'y'z'=0$
as long as $a_i,b_i, c_i,\,i=1,2,3$ are generically chosen.
(To see the intersection of $\ell_2=\ell_3=0$ in  $(x',y',z')$ coordinates of $\mathbb P^2$, we may 
assume that $\ell_1=1$ on that point.)  As $H_3(0,0,1)=h_{3}(a_3,b_3,c_3)\ne 0$. 
Similarly $H_3(1,0,0)$ and  $ H_3(0,1,0)\ne 0$. This means $ H_3$ is a convenient polynomial.
\end{proof}
Put $\al_1,\al_2,\al_3$ be the coefficients of ${x' }^4,{y'}^4, {z'}^4$ in $H_4(x',y',z'):=h_4(\ell_1,\ell_2,\ell_3)$.
Now we consider the toric modification $\hat \pi: X\to \mathbb C^3$ with respect to
$\Si^*$ with vertices $\{e_1,e_2,e_3,P\}$ where $ P={}^t(1,1,1)$ and the coordinates are $(x',y',z')$. The exceptional divisor $E(P)$ has three $A_1$ singularities at $\rho_i, i=1,2,3$. Take the toric chart $\si=\Cone\,(P,e_2,e_3)$. Let $\mathbf w=(w_1,w_2,w_3), w_1=u_{\si 1}$ be an admissible coordinate at $\rho_i$
so that 
\[\begin{split}
\hat H&=H(u_{\si 1},u_{\si 1}u_{\si 2},u_{\si 1}u_{\si 3})\\
&\equiv u_{\si 1}^3\left(H_3(1,u_{\si2},u_{\si3})+\al_1 u_{\si 1}+R)\right )\\
&=w_1^3(w_2^2+w_3^2+\al_1w_1+R)
\end{split}
\]
where $R\in (w_1^2)$.
Note that zeta function of $w_1^3(w_2^2+w_3^2+\al_1 w_1+R)$ is  
determined by 
$w_1^3(w_2^2+w_3^2+\al_1 w_1)$  whose zeta function is
$(1-t^4)^{-1}$.
Thus $H$ is an almost non-degenerate function in the coordinates $(x',y',z')$. 
This function is the same as the function considered in Example 2, \S 3.4 \cite{Almost}. We apply Theorem \ref{main1}, \cite{Almost} to get
%As $\zeta^{(s)}(t)\zeta^{er}(t)=1$,  we have 
$\zeta_{H}(t)=(1-t^4)^{-3}$ which proves the assertion.
\end{proof}
%%%%%
 { Subcase 2-1-b.} Suppose $h_3=0$ is a 3 lines which intersect  at a point in $\mathbb P^2$. We may assume that 
$h_3=x y (x+ay),\,a\ne 0$ (after a linear change of coordinate). 
Take a toric modification with  an admissible regular simplicial subdivision $\Si^*$.
It has a vertex $P={}^t(1,1,\al),\,\al\ge 1$ in $\Si^*$ which is adjacent to $e_3$ and $h_P=h_3$. This vertex gives $d(P)=3$ and  it gives factor
$(1-t^3)$ in the zeta function by A'Campo formula.
(In the Varchenko formula, this corresponds to $\zeta_I(t)$ with $I=\{1,2\}$.)
We can see no other vertices $\Si^*$ contribute the factor $(1-t^3)$. After further blowing ups,
no exceptional divisors appears with multiplicity $3$.  This is a contradiction to the assumption.
%%%%
\nl\indent
{Subcase 2-1-c.} $h_3=0$ are two lines where one line is doubled. Then we may assume that $h_3=x^2y$. 
If $\Ga(h)$ has a face of dimension 2 of degree 4, the only possibility is $h_4(x,0,z)$ has 1 dimensional support and with $x^2y$, it generate a face of dimension 2 with weight vector $P={}^t(1,2,1)$.
1-dimensional faces can be on $h_4(x,0,z)$ and $h_4(0,y,z)$.
We consider again degree 4 component $h_{4}$. Let $I=\{1,3\}$. If $h_{4}^I$ is 0 or a single monomial $x^az^{4-a}$ 
with $1\le a\le 3$, there are no possible degree 4 face of dimension $2$.
 If $h_4^I(x,z)$ is not  a monomial, assuming $x^4$ appears in $I$, we consider $\si=\Cone\,(P,e_2,e_3)$  where $P={}^t(1,2,1)$. 
 Let $\nu_I$ be the number of non-zero distinct roots of $f_{4}^I(x,z)=0$  as before.
%$(1,2,0)$ and the support polyhedron span the unique face with degree $4$. 
 Then $x=u_{\si1},y=u_{\si 1}^2u_{\si 2},z=u_{\si 1}u_{\si 3}$ and 
$E(P)^*$ is defined by
$u_{\si 2}+f_{4}^I(1,u_{\si 3})=0$. Thus $\chi(\hat E(P)^*)=-\nu_I$ and $\chi(E(P)_I^*)=\nu_I$.
 If $z^4$ appears in $h_4^I$, we consider in the chart $\Cone\,(P,e_1,e_3)$. If neither $x^4$ nor $z^4$ appears in $h_4^I$, we take 
 an admissible simplical cone
 $\tau=\Cone\,(P,Q,e_2)$ where $Q={}^t(a-1,b,a)$ with $a\gg 1$  and $b\gg a$ and the similar argument works as in Case 2-1-2. Anyway the contribution from  $h_4^I$ is cancelled with contribution from $E(P)^*$. However for $J=\{2,3\}$, $h^J(y,z)$ contribute to the zeta function by 
 $(1-t^4)^{\nu_J}$. On $\{x,y\}$ planes, there are no degree $4$ edges.

 Let $\de$ be the number of monomials among $\{x^4, y^4, z^4\}$ which appears in $h_4(x,y,z)$. Each monomial contribute by the factor $(1-t^4)^{-1}$ in the zeta function of $h$ and altogether we get  
$(1-t^4)^{-\de}$. Altogether the contribution to the zeta function on the factor  is $(1-t^4)^{-\de+\nu_J}$. As $-\de+\nu_J >-3$,  we get a contradiction to the assumption.

\indent
{ Subcase 2-1-d.} $h_3=0$ is a line with multiplicity 3. We assume that $h_3=x^3$. It is easy to see that the only possible 
%there does not exists 
 vertex $P$ with $d(P)=4$ which satisfies one of the conditions in Lemma \ref{essential vertex} is  $P={}^t(a,1,1),\,a>1$ and $h_P(x,y,z)=h_4(0,y,z)$.
Let $\de$ be the number of monomials  $\{y^4,z^4\}$ in $h_4(0,y,z)$
and let $\nu$ be the number of non-zero roots of $h_4(0,1,z)=0$.
Then the contribution to zeta-multiplicity factor is $(1-t^4)^{\nu-\de}$.
As $-1\le\nu-\de\le 2$, this is a contradiction to the assumption.
Alternatively we can also have a contradiction by showing that $x^3$ gives a factor $(1-t^3)^{-1}$.

\vspace{.3cm}\noindent
{\bf Subcase 2-2}. {\em Suppose $h_3=0$ is a union of a conic $C$ and a line $L$.}

\noindent\indent
{Subcase 2-2-1}. {\em The conic $C$ and the line $L$ are transversal
in $\mathbb P^2$.}
%Assume that 
%$h_3=j_2(x,y,z)(ax+by+cz)$ where $j_2$ is the defining polynomial of the conic component. 
We may assume that $h_3$ is convenient and two $A_1$ is not on  $xyz= 0$.
After one toric blowing up with vertices $\{e_1,e_2,e_3, P\}$  with $P:={}^t(1,1,1)\}$, we see that the exceptional divisor $E(P)\subset\hat E(P)$ has $2 A_1$ singularities  at the intersection of the conic and the line component and we see that the divisor $\hat E(P)$ gives zeta-multiplicity factor
$(1-t^3)$. This is a contradiction to the assumption.

\indent
{ Subcase 2-2-2}. {\em The conic $C$ and the line $L$ are tangent in $\mathbb P^2$.} Put $p=C\cap L$. Then the singularity of $(C\cup L,p)$ is isomorphic to $A_3$. Using the same toric modification,   we may assume that 
$E(P)$ has one $A_3$ singularity. This divisor does not gives any zeta factor %as the factor $(1-t^3)$ has exponent 0, 
as we see below.
Note that the zeta function of a non-singular cubic is $(1-t^3)^{-3}$ and by Theorem \ref{main1}, in the zeta function of $h$, the exponent of zeta factor becomes $-3 +\mu(A_3)=0$.
Consider the toric chart $\si=\Cone (e_1,e_2,P)$. Then the pull-back of $h$ is written as 
\[
\hat h(\mathbf u_\si)=u_{\si 3}^3\left (h_3(u_{\si1},u_{\si 2},1)+u_{\si 3}R(\mathbf u_{\si})\right)
\]
where $u_{\si 3}R$ is coming from higher terms of $h$.
Taking an admissible coordinates $(w_1,w_2,w_3),\,w_3=u_{\si 3}$ at the singular point
$\rho=(\al_1,\al_2,0)$,
we can write
\[
\hat h(\mathbf w)=w_3^3\left( w_1^2+w_2^4+w_3R(\mathbf w)
\right)
\]

 If the multiplicity of $w_3R(\mathbf w)$ is greater than or equal to 2, the multiplicity of $\hat h$ at the singular point is greater than or equal to $5$ and no factor $(1-t^4)$ appears in the zeta function. Only possible case is when  
$w_3R(\mathbf w)\equiv a w_3, a\ne 0$ modulo higher terms. Then  $\hat h(\mathbf w)$ is non-degenerate and by Theorem \ref{main1}, $\zeta_h(t)=(1-t^{16})^{-1}(1-t^8)(1-t^4)^{-1}$.
(For example, we can take as $h$, $(x^2+y^2-2 z^2)(x+y-2z)+z^4$.)
 This is also contradiction to the assumption.

\vspace{.3cm}\noindent
{\bf Subcase 2-3}.  Suppose that $h_3=0$ is an irreducible cubic.
We can choose a generic linear coordinates so that $h^{(3)}(x,y,z)$ is convenient and a possible singularity is one $A_1$ or one $A_2$. By Theorem \ref{main1}, the zeta-multiplicity factor is one of 
$(1-t^3)^{-3},\,(1-t^3)^{-2},\,(1-t^3)^{-1}$ according to  the cubic is non-singular, or one $A_1$ or   one $A_2$.
This is also a contradiction.

\noindent
\subsection{Case 3} {\em  Multiplicity 4.} We consider the last case. Assume that  $h=h_s$ has multiplicity 4. 
We divide this case in  two cases by the geometry of the curve $h_{s4}=0$.
\nl
(3-1) $h_{s4}=0$ has  non-isolated singularity for some $s$.
\nl
(3-2) $h_{s4}=0$ has only isolated singularity for any $s$.

We will show that the only possible case is (3.2).
%First we assert that $h_{4}=0$ has only isolated singularity. %  for any $s$.
\nl
{\bf Subcase} (3-1). Suppose it has non isolated singularity for some $s$.   Then  the possibility of $h_4=0$ are  
\nl
(a) one line with multiplicity 4, or
\nl
 (b) two lines where one line has multiplicity 3,  or \nl
 (c) two lines  of multiplicity 2, or \nl
  (d) one line with multiplicity 2 and two other lines, or
  \nl
(e)  one line with multiplicity 2 and an irreducible conic,
or
\nl (f) one conic  with multiplicity 2.

%\noindent
{ Subcase (3-1-a)} We assume the case (a).   Choose a linear coordinates so that $h_4=x^4$. Then it is easy to see that there are no faces of dimension $2$ or $1$ with degree $4$.
 The only possible  effective vertex with multiplicity $4$ is of the form $P={}^t(1,a, b)$ corresponding to the monomial $x^4$ and the contribution is $(1-t^4)^{-1}$. This is a contradiction to the assumption.
%multiplicity of the exceptional divisors are $5$ or bigger.

\indent
{ Subcase (3-1-b)} Consider the case (b).  We assume that $h_4=x^3y$.  In this case, it  is  impossible to  have an effective  exceptional divisor of multiplicity 4.
See Lemma \ref{essential vertex}.
\nl\indent
{ Subcase (3-1-c)} Assume that $h_4=x^2y^2$.  The same reason as the case $(b)$.\nl
\indent
{ Subcase (3-1-d)} Assume that $h_4=0$ has three lines $L_1, L_2, L_3$ where $L_1$ has  multiplicity 2.
After one point blowing-up $\hat \pi:X\to \mathbb C^3$,
$\hat E(P)\cong \mathbb P^2$ and $E(P)$ is a union of  three lines $L_1,L_2,L_3$ where $L_1$ has multiplicity 2. We can assume that $h_4=xyz^2$ or $x^2y(x+ay),\,a\ne 0$.
In the case $h_4=xyz^2$, there are no possibility of vertex $P$ with $d(P)=4$ and which contribute to the zeta function. Assume that $h_4=x^2 y(x+ay)$, the only possible vertex take the form $P={}^t(1,1,b), b\ge 1$ which is adjacent to $e_3$ and $h_P=h_4$. If this is the case, its contribution is $(1-t^4)$. This is also a contradiction to the assumption. (This case, we can also do the same discussion as Case (3-1-e) below.)
%

%\nl
{Subcase (3-1-e)} We can assume that $h_4(x,y)=x^2j_2(x,y,z)$ with $j_2$ is a smooth conic.  We take a toric modification $\hat\pi:X\to \mathbb C^3$ which respect to $\Si^*=\{e_1, e_2,e_3,P\}$, $P={}^t(1,1,1)$.  $E(P)$ is a union of smooth conic $C$ and a line of multiplicity 2. 
Using the toric chart $\si:=\Cone(e_1,e_2,P)$, the pull-back of $h$ takes the form  
\[\begin{split}
h(\mathbf u_\si)&=h_4(\mathbf u_\si)+h_\ell(\mathbf u_\si)+\text{(higher degree terms)},\,\ell\ge 5\\
\hat h( u_{\si 11}, u_{\si 2}, u_{\si 3})&= u_{\si 3}^4\left(
 u_{\si 1}^2j_2( u_{\si 1}, u_{\si 2},1)+ u_{\si 3}^{\ell-4}h_\ell( u_{\si 1}, u_{\si 2},1)+\dots
\right)
\end{split}
\]
$C\cap L$ is either 2 points or one point. 
Put them $\rho_1$ and $\rho_2$ be the intersection of $C$ and $L$ in $E(P)$. In the latter case, $L$ is tangent to $C$ and $\rho_2=\rho_1$.
On $L$ and $C$, there are finite points such that the function $\hat h$ is not equi-singular  along $L$ or $C$.  These exceptional points on $L$ are $\rho_1,\rho_2$ and the points in  the intersection  $ u_{\si 1}=h_\ell( u_{\si 1}, u_{\si 2},1)=0$ on $L$  and $j_2( u_{\si 1}, u_{\si 2},1)=h_\ell( u_{\si 1}, u_{\si 2},1)=0$ on $C$ respectively. 
%(If $h=h_4+h_\ell+\text{(higher terms)}$, we replace $h_5$ by $h_\ell$.)
Put them $\rho_3,\dots, \rho_k$.
Take a small $\eps$ ball $B_{\eps}(\rho_i)$ centered at $\rho_i$ and put $B=\cup_{i=1}^k B_\eps(\rho_i)$.
Let $N(\hat E(P))$, $N(L)$ and $N(C)$ be the sufficiently small controlled tubular neighborhoods of  $\hat E(P)\setminus (L\cup C)$, $L\setminus ( L\cap B)$ and $C\setminus (C\cap B)$. Put $N=N(L)\cup N(C)$.
Let $N(\hat E(P))'=N(\hat E(P))\setminus (B\cup N)$.
We divide the Milnor fibration of $\hat h$  into fibrations on
$N(\hat E(P))', N(L), N(C)$ and $B_\eps(\rho_i), i=1,\dots,k$
and carry out  the exact same argument as that  in \cite{Almost}. Note that $\chi(N(\hat E(P))')=1$ or $0$, according $\rho_1\ne \rho_2$ or $\rho_1=\rho_2$ and $\chi(L\cup C)=2$ or $3$. The zeta function of $\hat h|_{N(\hat E(P))'}$ is 
$(1-t^4)^{-1}$ or $(1-t^4)^0=1$. The contribution of the zeta function  $\hat h|_{N(L)}$ to the zeta-multiplicity factor $(1-t^4)$ is trivial  as the normal zeta function is described  by the $ u_{\si 3}^4( u_{\si 1}^2+ u_{\si 3}^{\ell-4})$ (Sublemma 4,\cite{Almost}).
Similarly the restriction of the Milnor fibration $\hat h$ on $N(C)$ does not contribute to the zeta multiplicity factor as the normal zeta function corresponds to
$ u_{\si 3}^4(\tilde j_2+ u_{\si 3}^{\ell-4})$ where $( u_{\si 3}, \tilde j_2)$, $\tilde j_2=j_2(u_{\si 1},u_{\si 2},1)$ is coordinates of the normal slice.  That is, there is a local coordinates $(u_{\si 3}, \tilde j_2, \exists v_1)$ locally where $v_1$ is a local coordinate of $C$.
%This has only $(1-t^5)$ factor.
 To get  the zeta function of  the restriction of Milnor fibration on $B\setminus \hat h\inv(0)$, we have to take further resolution. However over $\rho_i$, we get 
 exceptional divisors of multiplicity greater than or equal to $5$, as the multiplicity 
 of $\hat h$ at $\rho_i$ is greater than or equal to $5$. Over $N(L)$ and $N(C)$, the multiplicity of $\hat h$ is $6$ and $5$ respectively.
Combining these data, the possible zeta-multiplicity factor  in $\zeta_h(t)$ is either 
$(1-t^4)^{-1}$ or $1$, a contradiction to the assumption.

{ Subcase (3-1-f)} Assume that $h_4=0$ is non-reduced conic $C$ of multiplicity $2$. We assume that $h_4=j_2^2$ and  $h=j_2^2+h_\ell+\text{(higher terms)}$ as above.
 After one point blowing-up $\hat \pi:X\to \mathbb  C^3$, the exceptional divisor is $\hat E(P)=\mathbb P^2$ and 
 $E(P)$ is a non-reduced conic. Use again the toric chart $\si=\Cone(e_1,e_2,P)$ as above.
 Then
 \[
 \hat h(\mathbf u_\si)=u_{\si 3}^4\left(j_2(u_{\si 1},u_{\si 2},1)^2+u_{\si 3}^{\ell-4}h_\ell(u_{\si 1},u_{\si 2},1)+\text{(higher terms)}
 \right).
 \]
 Let $\rho_1,\dots, \rho_k$ be the intersection of 
 $j_2(u_{\si 1},u_{\si 2},1)=h_\ell(u_{\si 1},u_{\si 2},1)=0$ and take a small disk $B_\eps(\rho_i)$ centered at $\rho_i$ for each $i=1,\dots,k$ and put $B=\cup_{i=1}^k B_\eps(\rho_i)$.
 Take a controlled tubular neighborhoods $N(\hat E(P)\setminus C)$,  $N(C)$ of $C=E(P)\setminus B$ and $N(\hat E(P))$ be a tubular neighborhood of $\hat E(P)\setminus (B\cup N(C))$.
 We divide Milnor fibration into the following  parts. The complement $N(\hat E(P))':=N(\hat E(P))\setminus (N(C)\cup B)$ and $N(C)\setminus (N(C)\cap B)$ and $B$.   We do the same discussion as in Subcase (3-1-e). Thus $N(\hat E(P))'$ contribute to the zeta function by
 $(1-t^4)^{-3+2}=(1-t^4)^{-1}$. 
 On $N(C)$, the normal zeta function is described by
 $u_{\si 3}^4(j_2(u_{\si 1},u_{\si 2},1)^2+u_{\si 3}^{\ell-4}(u_{\si 1},u_{\si 2},1)+\text{(higher terms)}$ 
 which is equivalent to 
 $u_{\si 3}^4(\tilde j_2^2+u_{\si 3}^{\ell-4})$ with $(u_{\si 3},\tilde j_2)$ are coordinates of the normal slice. Thus it contribute for the factor $(1-t^4)$  trivially. On $\rho_i$, $\hat h$ has multiplicity $7$ and also this part also gives  nothing for the zeta-multiplicity factor. Thus we conclude that the case (3-1-f) is also not possible.
 
 \vspace{.2cm}\noindent
 {\bf Subcase 3-2 (Last case)}. Assume that the family $h_s=0$ has only isolated singularity 
 and multiplicity is $4$ and the total and local  Milnor numbers of $h_{s4}$ must be constant for any $0\le s\le 1$ by Theorem \ref{main1}.  
 This implies $h_s$ is a $\mu^*$-constant family from $f_1$ to $f_2$ or from $f_3$ to $f_4$. However 
 by  Lemma \ref{mu*-zariski}, 
 %the degree $4$ part gives a $\mu^*$-constant family of
 % quartics with 3 $A_1$ or $4A_1$. This gives a family of projective curves  from
 %$q_1=0$ to $q_2=0$ or from $q_3=0$ to $q_4=0$. However by the choice of the pair $(q_1,q_1)$ and $(q_3,q_4)$, 
 this is impossible and the proof of Theorem \ref{main1} is completed.
 %%%%%%%%%%%%%%
\section{Geometric structure of the links $K_{f_1}$ and $K_{f_2}$}
In this section, we study further geometric structure of the link 3-manifolds $K_{f_i}, i=1,2$.  As second  result, we will   show that they are not diffeomorphic, though their zeta functions are equal.
Recall that 
\[\begin{split}
f_1(x,y,z)&=q_1(x,y,z)+z^{m+4},\\
f_2(x,y,z)&=q_2(x,y,z)+z^{m+4}. %,\, q_2=q_{2,3}q_{2,1}\\
\end{split}
\]
where $q_1$ is a quadric with 3 $A_1$ singularities and $q_{2}$ is a union of a smooth cubic and a generic line.
We assume that all of these polynomials are convenient.
Let $f(\mathbf z)$ be one of $f_1$ or $f_2$. First we take a toric modification $\hat \pi: X\to \mathbb C^3$
with respect to the vertices
$\{e_1,e_2,e_3,P\}$ with $P={}^t(1,1,1)$. Take the coordinate chart $\xi=\Cone\,(e_1,e_2,P)$
with coordinates $\mathbf u_\xi=(u_{\xi 1},u_{\xi2},u_{\xi3})$. The pull back of $f$ is given  as
\begin{eqnarray}\label{left1}
\hat f_i(\mathbf u_{\xi}):=\hat\pi^*f_i(\mathbf u_{\xi})&=
u_{\xi 3}^4\left( q_i(u_{\xi 1},u_{\xi 2},1)+u_{\xi 3}^m\right ), i=1,2
%&=v_3^4(v_1^2+v_2^2+v_3^m)
\end{eqnarray}
as $(x,y,z)=(u_{\xi 1}u_{\xi 3},u_{\xi 2}u_{\xi 3},u_{\xi 3})$.
%To distinguish from other local coordinates, we write $u_{\al i}=u_i$ hereafter.
%Thus\[\hat f_i=u_3^4(q_i(u_1,u_2,1)+u_3^{m+4}).\]

Let $\rho_\al,\, \al=1,2,3$ be the singular points of $E(P)$. Take 
 $(v_{\al 1},v_{\al2},v_{\al3})$   admissible coordinates in the neighborhood $U_\al$ centered at $\rho_\al$ so that $v_{\al 3}=u_{\xi3}$ and 
 %$v_{\al,3}$ is  the restriction of the function $u_{\xi3}$ defined on $X$. 
 %
 %Then  we have the expression:
 \begin{eqnarray}
 \hat f_i(\mathbf v)=u_3^4(v_{\al1}^2+v_{\al2}^2+u_3^{m}).
 \end{eqnarray}
 To distinguish from the local coordinates, we write $v_{\al3}=u_{\xi3}$  as $u_3$.
 For $f_2$, we   consider  also another coordinates.
 Let $q_{2,3}, q_{2,1}$ be the defining polynomial of the cubic and linear component of $q_2=0$. Thus $q_2=q_{2,3}q_{2,1}$. We define $w_1:=q_{2,3}(u_{\xi1},u_{\xi2},1)$ and $w_2:=q_{2,1}(u_{\xi1},u_{\xi2},1)$.  
 Note that 
 \[\begin{split}
 q_{2,3}(u_{\xi 1}u_{\xi 3},u_{\xi 2}u_{\xi 3},u_{\xi 3})&=u_{\xi 3}^3\,q_{2,3}(u_{\xi1},u_{\xi2},1 )=u_{\xi 3}^3w_1,\\
 q_{2,1}(u_{\xi 1}u_{\xi 3},u_{\xi 2}u_{\xi 3},u_{\xi 3})&=u_{\xi 3}\, q_{2,1}(u_{\xi1},u_{\xi2},1)=u_{\xi 1}w_2
 \end{split}
 \]
 and $q_{ 2,i}(u_{\xi1},u_{\xi2},1)=0$ with $i=3,\, 1$  are the defining polynomials of the strict transform of the cubic and the line component respectively. As the cubic and the line intersect transversely, $(w_1,w_2)$ is a local coordinate of $E(P)$ 
 and 
 $(w_1,w_2,u_3)$ is a local coordinate of $X$ 
 in the neighborhood of $\rho_\al,\,\al=1,2,3$ for $f_2$
 %and  we can assume that $(w_1,w_2,u_3)$  is  analytic coordinates on $U_\al$
  so that 
% $w_1:=q_{2,3}(u_{\xi1},u_{\xi2},1),\, w_2:=q_{2,1}(u_{\xi1},u_{\xi2},1)$ and  
the pull-back of $\hat f_2(\mathbf u_\xi)$ is now given as 
\[ \hat f_2(w_1,w_2,u_3)=u_3^4(w_1w_2+u_3^m).
 \]
 Thus the local coordinates $(v_{\al1},v_{\al2})$ for $f_2$ are  chosen so that  they satisfy
 \begin{eqnarray}
 v_{\al1}+\sqrt{-1}v_{\al2}=w_1,\quad v_{\al1}-\sqrt{-1}v_{\al2}=w_2.
 \end{eqnarray}
 $w_1, w_2$ is also globally defined on the toric chart $U_\xi$
(and also on $X$ as a meromorphic function).
 As is obvious from the expression, $\hat f_i$ is weighted homogeneous in $\mathbf v_\al$ and  the dual Newton diagram $\Ga^*(\hat f_i;\mathbf v_\al)$ at $\rho_\al$ has only one positive vertex   $R_\al={}^t(m,m,2)$ or $S_{m_0}={}^t(m_0,m_0,1)$ for $m$ odd or even ($m=2m_0$) respectively$e_1,e_2,e_3$. $R_\al$ (or $S_{m_0}$) is the weight vector of $\hat f_i(\mathbf v_\al)$.
 
% \end{comment}
  %%%%%%%%%%%%
 \subsection{Regular simplicial subdivision $\Si_\al^*$}
Suppose $m$ is an odd integer and put $m=2m_0+1$.
The regular simplicial cone subdivision $\Si_\al^*$ is given as the left $\Si_o^*$ of   Figure \ref{Figure1}.
Here $R_\al={}^t(m,m,2),\,T_\al={}^t(m_0+1,m_0+1,1)$ for $m=2m_0+1$
and $m_0$  vertices $S_{\al,1},\dots, S_{\al,m_0}$   are added where $S_{\al,i}={}^t(i,i,1),\,i=1,\dots,m_0$.
For an even  $m=2m_0$, we do not need $T_\al$ and $R_\al={}^t(m_0,m_0,1)$.
% and  $S_{m_0}$ is the weight vector of $\hat f$.
 See  the right subdivision $\Si_e^*$ of Figure \ref{Figure1}.  In the following, we consider the case $m=2m_0+1$ first.  The case $m=2m_0$ is similar.
 \begin{figure}[htb]  %ebb *.pdf
\setlength{\unitlength}{1bp}
\begin{picture}(600,300) (0,-100) %(-100,-100)
{\includegraphics[width=12cm, bb=0 0 470 182]{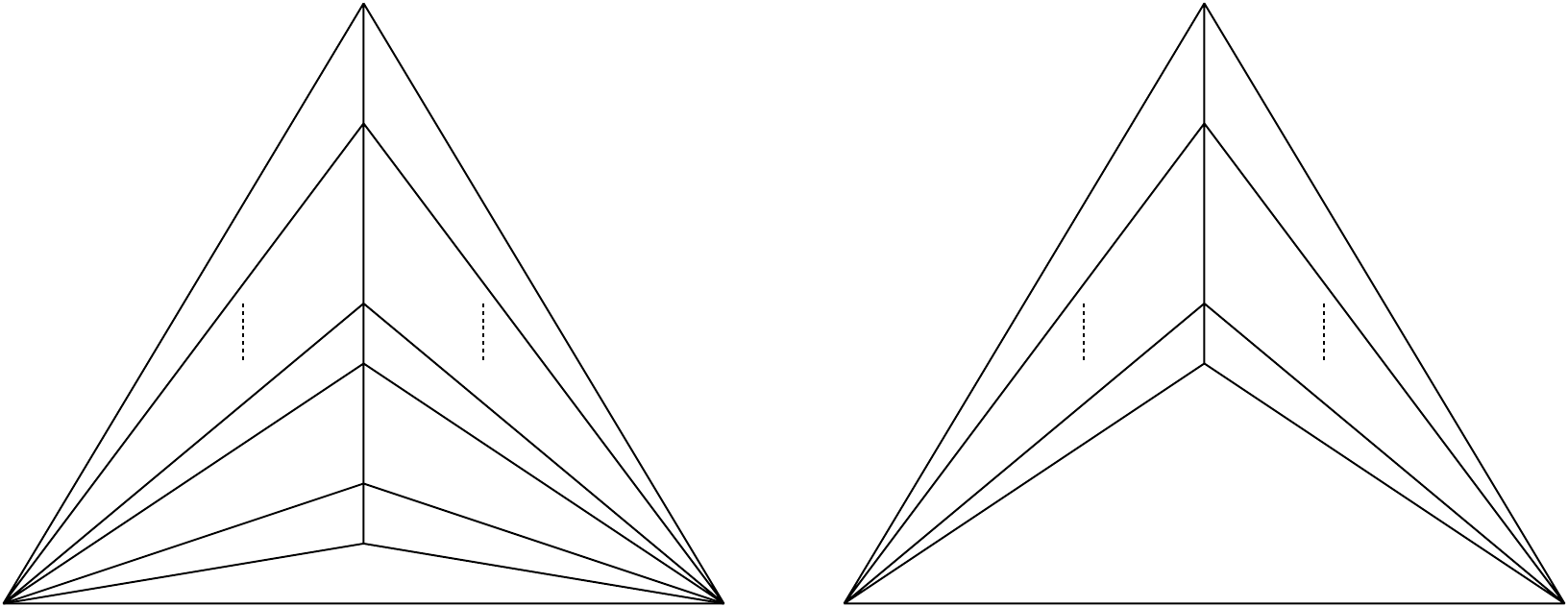}}
%\put(-320,165){${\hat f}\inv(\pm\de)$: green line}
%\put(-390,113){$S_{\al,a}$}
\put(-274,104){${\text{\tiny S}}_1$}
\put(-261,104){\circle*{3}}
\put(-90,103){${\text{\tiny S}}_1$}
\put(-79,105){\circle*{3}}
\put(-290, -10){$\Si_{o}^*,\,m=2m_0+1$}
\put(-90,-10){$\Si_{e}^*,\,m=2m_0$}
\put(-270,6){{\tiny T}}
\put(-261,13){\circle*{3}}
\put(-272,40){${\text{\tiny S}}_{{\text{\tiny m}}_0}$}
\put(-261,53){\circle*{3}}
\put(-261,65){\circle*{3}}
\put(-84,40){${\text{\tiny S}}_{{\text{\tiny m}}_0}$}
\put(-79,52){\circle*{3}}
\put(-79,65){\circle*{3}}
\put(-270,26){{\tiny R}}
\put(-261,26){\circle*{3}}
\put(-272,134){$e_3$}
\put(-261,131){\circle*{3}}
\put(-77,134){$e_3$}
\put(-79,132){\circle*{3}}
\put(-498,5){$e_{\al,1}$}
\put(-350,0){$e_{1}$}
\put(-339,0){\circle*{3}}
\put(-182,0){$e_2$}
\put(-184,0){\circle*{3}}
\put(-168,0){$e_1$}
\put(-157,0){\circle*{3}}
\put(2,0){$e_2$}
\put(-0,0){\circle*{3}}
\end{picture}
\vspace{-3cm}
\caption{ $\Si_o^*:m=2m_0+1$:odd,\, 
$\Si_e^*:m=2m_0$:even}\label{Figure1}
\end{figure}
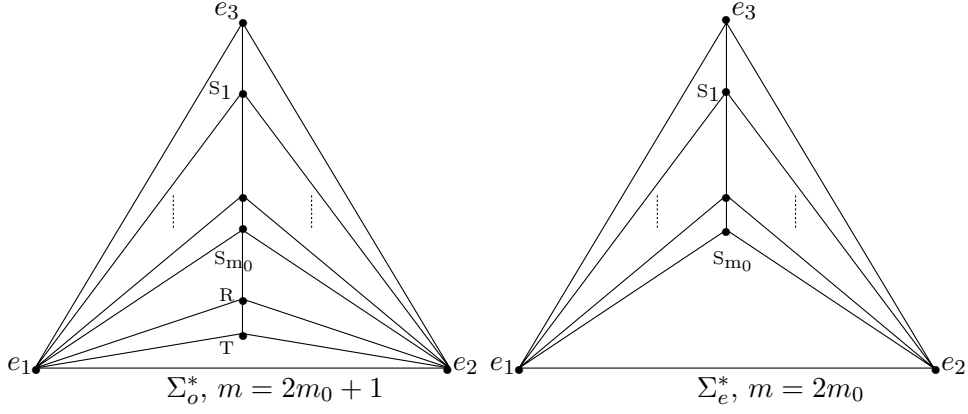

%For the  case an even $m$, the argument   is similar and we leave this case to the reader. 
 
 To see the manifold structure, we consider 
a toric modification 
$
\omega_\al : Y_\al\to X, \quad \al=1,2,3
$ with respect to $\Si_o^*$ or $\Si_e^*$. ($\Si_o^*,\,\Si_e^*$ is the same for every $\rho_\al.$)
Three modification  can be canonically glued together to get a final resolution
$\omega: Y\to X$ and   by taking composition with $\hat \pi:X\to \mathbb C^3$,
we get a good resolution of $f_i$ restricting $\Pi: Y\to \mathbb C^3$
to an open neighborhood $\mathcal U_0$ of the origin. The exceptional divisors of $\Pi$ are all compact and 
they are $\hat E(P)$ (from $\hat \pi$) and $\hat E(R_\al),\hat E(S_{\al i}),\,i=1,\dots, m_0$  from $\omega_\al,\,\al=1,2,3$, and  $\hat E(T_\al)$ for an odd $m=2m_0+1$. (If $m=2m_0$, the exceptional divisors are
$\hat E(S_i),\, i=1,\dots, m_0$.)
Let $\tilde V_\be$ be the strict transform of $V_\be=f_\be\inv(0)$ to $Y$. Hereafter $\al=1,2,3$ are the indices of the singular points and $\be=1,2$ are the choice of functions.
Recall that for a vertex $K\in \Si^*$, the restriction of the exceptional divisors $E(K)$ to $\tilde V_\be$ are non-empty if the supporting face $\De(K)$ has dimension greater than or equal to 1. Thus $E(T_\al)$ is empty as $f_{\be,T_\al}=u_3^{m+4}$.
%(the face function of $f_\be$ for the weight $T_\al$).
\begin{Remark}{\rm  To be precise, the resolution space $X$ and  $Y$ for $f_1$ and for $f_2$ are different complex spaces and it is better to be distinguished
and to be written as $\hat \pi_1:X_1\to \mathbb C^3$, $\hat \pi_2:X_2\to \mathbb C^3$ and 
$\Pi_1:Y_1\to \mathbb C^3$ and $\Pi_2:Y_2\to \mathbb C^3$.
Also the exceptional divisors $S_{\al,i}, i=1,\dots, m_0$ and $R_\al$ are sitting in the different space $Y_1$ and $Y_2$. Thus it is more precise to write them as 
$S_{\al,i,j}$ and $R_{\al,j}$ for $j=1,2$.
However except the central divisor $P$, they are the same Riemann surfaces and the resolution graphs are isomorphic and  the link 3-manifolds are determined by the dual  resolution graphs. Therefore we ignore this too precise notations
 (with too many suffixes)  and we use  the same letter for the both cases, unless any confusion is likely.}
\end{Remark}

 We denote the strict transform of the exceptional divisor of $\hat \pi_\be:X_\be\to \mathbb C^3$, $ E(P)\subset X_\be$, to $Y$ by 
$P_{\be},\be=1,2$ respectively, as they are topologically  different.
\begin{Proposition} For $f_1$,  $P_1$ is a smooth rational curve. For $f_2$,  $P_2$ is a union of smooth cubic and a line $P_2=P_{2,3}+P_{2,1}$, where  $P_{2,3}$  and
$P_{2,1}$ are the strict transforms of the cubic and the line respectively.
\end{Proposition}
\begin{proof} By the discussion of Euler characteristics, we know that  $\chi( E(P))=
-4+3=-1$ for $f_1$. Recall  $E(P)$ is a  quartic with $3A_1$, $\chi(E(P))=-1$. $P_1$ is the normalization of $E(P)$ at three points $\rho_\al,\,\al=1,2,3$. In $Y_\al$, each singular point $\rho_i$ is separated in two points. Thus $\chi(P_1)=-1+3=2$. The second assertion is obvious from the assumption on $q_2$.
\end{proof}
Note that $P_{2,3}\cap P_{2,1}=\emptyset$ (after the modification $\omega$). By abuse of notation, we also  denote the exceptional divisors $E(S_{\al i})$ and $E(R)$ by  the same letter $S_{\al, i}$ and $R_\al$ for simplicity of the notations.
\begin{Proposition} (1)
$S_{\al i}$ has two components which are $\mathbb P^1$ for $1\le i\le m_0$ if $m=2m_0+1$
and for $1\le i\le m_0-1$ for $m=2m_0$. We can call each component as 
$S_{\al i}^{+},  S_{\al i}^{-}$ so that $S_{\al, i}^{\pm}\cdot S_{\al, i+1}^{\pm}=1$ 
and $S_{\al, i}^{\pm}\cdot S_{\al, i+1}^{\mp}=0$. 
\nl
(2) For $m=2m_0+1$, $R_\al$ is a rational sphere and $R_\al\cdot S_{\al, m_0}^\pm=1$ and $S_{\al,m_0}$ is a union of two rational spheres. For $m=2m_0$, $S_{\al,m_0}$ is connected and  a rational sphere.
\nl
(3) $P_j\cdot S_{\al 1}=2$. More precisely,
$P_1\cdot S_{\al 1}^{\pm}=1$ for $f_1$
and 
for $f_2$, $P_2=P_{2,3}+P_{2,1}$ and we have that 
$P_{2,3}\cdot S_{\al 1}^\pm=1, 0$ and $S_{\al 1}^\pm\cdot P_{2,1}=0,1$  respectively. % according $\eps=+$ or $-$.
\end{Proposition}
\begin{proof} As we are working on $Y_\al$ for $\al=1, 2$ simultaneously, we skip the suffix $\al$. % in this proof.
First consider the toric chart $\si_i:=\Cone (S_{i},S_{ i+1},e_{1})$ for $0\le i\le m_0-1$.
which corresponds to the unimodular matrix
\[
\left(\begin{matrix}  i&i+1&1\\i&i+1&0\\1&1&0\end{matrix}
\right)
\]
We understand $S_0=e_3$ in the above notation.
Denote the coordinates  of $\mathbb C_{\si_i}^3$ as  $(u_{i, 1},u_{i,2},u_{i,3})$ and 
\[
v_{\al1}=u_{i, 1}^{i}u_{i,2}^{i+1}u_{i, 3},\, v_{\al2}=u_{i, 1}^{i}u_{i,2}^{i+1},\,u_3=u_{i,1}u_{i,2}
\]
and pull back of $\hat f$ by $\omega_\al$ for $m=2m_0+1$ is given as 
\begin{eqnarray}\label{sigma-i}
\omega_\al^*{\hat f}
%&=&(u_{i,1}u_{i,2})^4\left((u_{i,1}^{i}u_{i,2}^{i+1}u_{i,3})^2+( u_{i,1}^{i}u_{i,2}^{i+1})^2+(u_{i,1}u_{i,2})^m\right)\\
&=&u_{i,1}^{2i+4}u_{i,2}^{2i+6}
\left(u_{i,3}^2+1+u_{i,1}^{m-2i}u_{i,2}^{m-2i-2}
\right),\, \, i<m_0-1.\\
\end{eqnarray}
The divisor $S_i$ and $S_{i+1}$  are  defined in this chart by $u_{i,1}=0$ and $u_{i,2}=0$ respectively and their two components $S_i^{\pm},\, S_{i+1}^{\pm}$ correspond to $u_{i,3}=\pm \sqrt{-1}$ respectively. 
In the chart  $\si_{m_0}=\Cone(S_{m_0}, R, e_1)$, 
\begin{eqnarray}\label{sigma-m_0}
\omega_\al^*{\hat f}&=&
u_{m_0,1}^{2m_0+4}u_{m_0,2}^{2m+8}\left\{u_{m_0,3}^2+1+u_{m_0,1}
\right\}.
\end{eqnarray} We can see $R$ is defined by $u_{m_0,2}=0$ and $R\cdot S_{m_0}=2$. This is described as $u_{m_0,1}=u_{m_0,3}^2+1=0$.
For an even $m$ with  $m=2m_0$, $\omega_\al^*{\hat f}$ is as above
(\ref{sigma-i}) for $i<m_0-1$ and for $i=m_0-1$,
%and for $m=2m_0$,
the above equation takes the form:
\begin{eqnarray}\label{sigma-ii}
%(m=2m_0,\,i=m_0-1)&:&\quad \notag\\
\omega_\al^*{\hat f}
&=&u_{m_0-1,1}^{2m_0+2}u_{m_0-1,2}^{2m_0+4}
\left(u_{m_0-1,3}^2+1+u_{m_0-1,1}^{2}
\right) \quad \text{if}\,\,m=2m_0.\notag
\end{eqnarray}
As $E(S_{m_0})=\{u_{m_0-1,2}=u_{m_0-1,3}^2+1+u_{m_0-1,1}^{2}=0\}$ if $m=2m_0$, we see that it is connected and a rational curve. Other part, the argument for $m$ odd or  even is exactly the same, and we do the argument for $m=2m_0+1$ hereafter.

In the chart $\mathbb C_{\si_i}^2$, $\hat E(S_{i})$ and $\hat E(S_{i+1})$ is defined by $u_{i,1}=0$ and $u_{i,2}=0$ respectively
and two components are $u_{i,1}=u_{i,3}\pm\sqrt{-1}=0$.
We define $S_{i}^+=\{u_{i,1}=u_{i,3}+\sqrt{-1}=0\}$ and $S_{i}^-=\{u_{i,1}=u_{i,3}-\sqrt{-1}=0\}$.
In the next chart $\si_{i+1}=\Cone\,(S_{i+1},S_{i+2},e_1)$ with coordinate
$(u_{i+1, 1},u_{i+1,2},u_{i+1,3})$, they are related by
\[u_{i+1,3}=u_{i,3},\quad u_{i,1}=u_{i+1,2}\inv,\, u_{i,2}=u_{\i+1,1}u_{i+1,2}^2.
\]
Thus we see $S_{i}^{\pm}\cdot S_{i+1}^{\pm}=1$.
For  $f_2$, in the chart $\si_0$, we have
\[
w_1=v_1+\sqrt{-1}v_2=u_{0,2}(u_{0,3}+\sqrt{-1}),
\quad w_2=v_1-\sqrt{-1}v_2=u_{0,2}(u_{0,3}-\sqrt{-1})
\]
where  $\si_0=\Cone\,(e_3,S_1,e_1)$ and
$v_1=u_{0,2}u_{0,3},v_2=u_{0,2},u_3=u_{0,1}u_{0,2}$. 
(Recall $w_1=0$ is the defining function of $P_{2,3}$ and $w_2=0$ defines $P_{2,1}$.)
This implies $S_{0}^+=P_{2,3}$ and $S_{0}^-=P_{2,1}$ as is desired.
Now we consider the last chart $\tau=\Cone\,(S_{m_0},R,e_1)$ with coordinates
$(u_{\tau1},u_{\tau 2},u_{\tau 3})$ (here $m=2m_0+1$) and 
\[
\begin{split}
&v_1=u_{\tau1}^{m_0}u_{\tau2}^m u_{\tau3},\quad v_2=u_{\tau1}^{m_0}u_{\tau2}^m,\quad
u_3=u_{\tau1} u_{\tau2}^2\\
&\omega^*{\hat f}=u_{\tau1}^{m}u_{\tau2}^{2m+8}(u_{\tau 3}^2+1+u_{\tau1})
\end{split}
\]
We see $\chi(E(R)^*)=-2$ and $\chi(E(R)\cap E(S_{m_0}^\pm))=1$ and 
$\chi(E(R)\cap E(e_1))=\chi(\hat E(R)\cap \hat E(e_2))=1$. Thus $\chi(E(R))=2$.
\end{proof}
%%%%%%%%%%5
\subsection{Resolution graph of $V_1, V_2$}\label{ResolutionV1V2}
Now we come to the crucial part.  Let $\Ga_\be$ be the dual resolution graph of $\Pi|_{\tilde V_\be}\, \tilde V_\be\to V_\be,\be=1,2$. Here $V_\be=V(f_\be),\,\be=1,2$.

First we consider $V_1$. $\Ga_1$ has $6m_0+4$ vertices corresponding to
\[
P_1, S_{\al 1}^\pm,\dots, S_{\al m_0}^\pm, R_\al,\,\al=1,2,3
\]
 and  $\Ga_1$ is three cycle graph centered at $P_1$. For each $\al$, one cycle centered at $P_1$ is this:
 \[
 \overset {P_1} \bullet\rule[1mm]{.5cm}{0.3mm} \overset {S_{\al1}^+} \bullet\rule[1mm]{.5cm}{0.3mm}  \dots\rule[1mm]{.5cm}{0.3mm} \overset {S_{\al m_0}^+} \bullet\rule[1mm]{.5cm}{0.3mm} \overset{R_\al} \bullet \rule[1mm]{.5cm}{0.3mm}\overset {S_{\al m_0}^- }\bullet \rule[1mm]{.5cm}{0.3mm} \dots\rule[1mm]{.5cm}{0.3mm} 
\rule[1mm]{.5cm}{0.3mm}   \overset {S_{\al 1}^-} \bullet\rule[1mm]{.5cm}{0.3mm}
\overset {P_1}\bullet
\]
%$C_1,E_{\al 1}^+,\dots, E_{\al m_0}^+,E(P),E_{\al m_0}^-,\dots, E_{\al 1}^-,C_1$.
Figure \ref{Figure2} show the graph $\Ga_1$ for $m=3$.

Now we consider the case $f_2$. The  central divisor $E(P)$ split into two vertices corresponding to $P_{2,3}$ and $P_{2,1}$ and $\Ga_2$ has  $6m_0+5$ vertices. There are three trees  from $P_{2,3}$ to $P_{2,1}$.   See  Figure \ref{Figure3}.
\[
 \overset {P_{2,3}} \bullet\rule[1mm]{.5cm}{0.3mm} \overset {S_{\al 1}^+} \bullet\rule[1mm]{.5cm}{0.3mm}  \dots\rule[1mm]{.5cm}{0.3mm} \overset {S_{\al m_0}^+} \bullet\rule[1mm]{.5cm}{0.3mm} \overset{R_\al} \bullet \rule[1mm]{.5cm}{0.3mm}\overset {S_{\al m_0}^- }\bullet \rule[1mm]{.5cm}{0.3mm} \dots\rule[1mm]{.5cm}{0.3mm} 
\rule[1mm]{.5cm}{0.3mm}   \overset {S_{\al1}^-} \bullet\rule[1mm]{.5cm}{0.3mm}
\overset {P_{2,1}} \bullet
\]
% Figure \ref{Figure3}
%\newpage
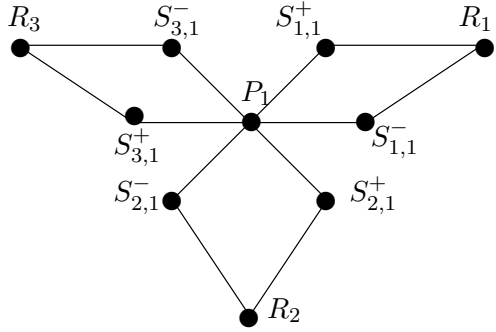
\begin{figure}[htb]  %ebb *.pdf
\setlength{\unitlength}{1bp}
\begin{picture}(600,300)(-100,-100)
{\includegraphics[width=17cm, bb=0 0 595 842]{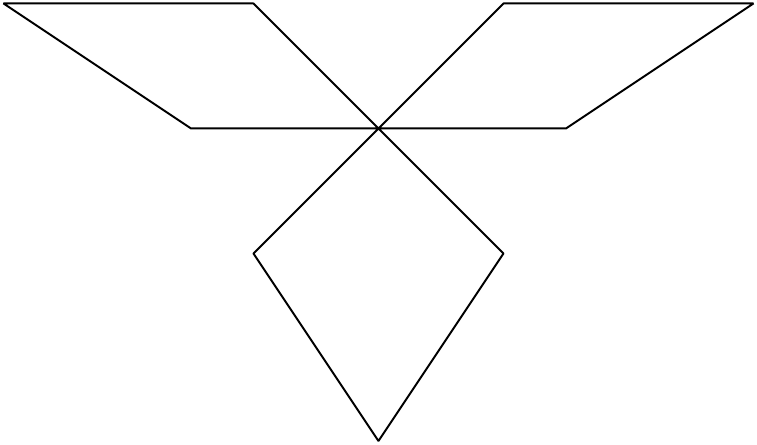}}
\put(-315,110){$R_1$}
\put(-305,102){\circle*{7}}
\put(-384,110){$S_{1,1}^+$}
\put(-365,102){\circle*{7}}
\put(-397,82){$P_1$}
\put(-393,74){\circle*{7}}
\put(-348,65){$S_{1,1}^-$}
\put(-350,74){\circle*{7}}
\put(-356,44){$S_{2,1}^+$}
\put(-365,44){\circle*{7}}
\put(-387,0){$R_2$}
\put(-394,0){\circle*{7}}
\put(-445,44){$S_{2,1}^-$}
\put(-423,44){\circle*{7}}
\put(-445,62){$S_{3,1}^+$}
\put(-437,76){\circle*{7}}
\put(-485,110){$R_3$}
\put(-480,102){\circle*{7}}
\put(-430,110){$S_{3,1}^-$}
\put(-423,102){\circle*{7}}
\end{picture}
\vspace{-3cm}
\caption{Graph $\Ga_1,\,m=3$}\label{Figure2}
\end{figure}
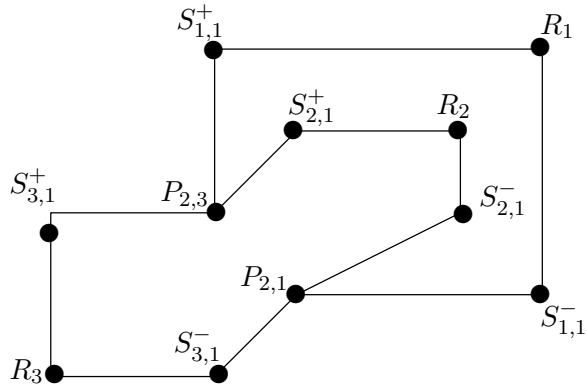
\begin{figure}[htb]  %ebb *.pdf
\setlength{\unitlength}{1bp}
\begin{picture}(600,300)(-100,-100)
{\includegraphics[width=18cm, bb=0 0 595 842]{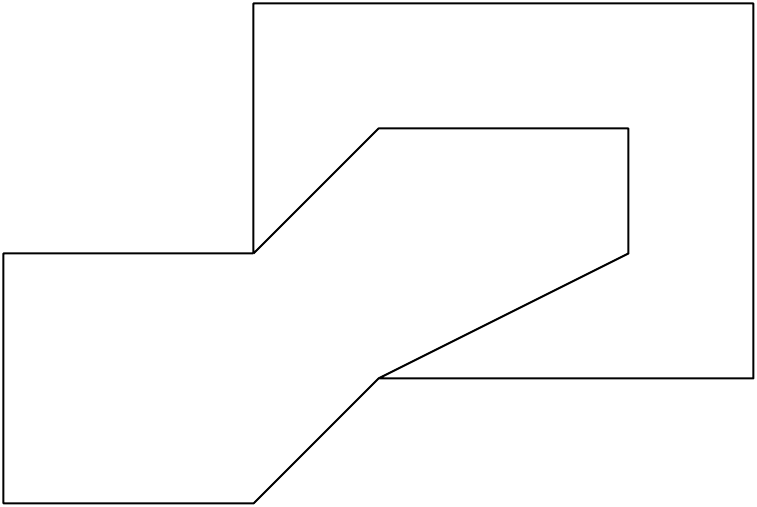}}
\put(-325,130){$R_1$}
\put(-325,125){\circle*{7}}
\put(-325,20){$S_{1,1}^-$}
\put(-325,32){\circle*{7}}
\put(-364,100){$R_2$}
\put(-356,94){\circle*{7}}
\put(-420,100){$S_{2,1}^+$}
\put(-418,94){\circle*{7}}
\put(-348,65){$S_{2,1}^-$}
\put(-354,62){\circle*{7}}
%\put(-356,40){$S_{2,1}^+$}
%\put(-368,47){\circle*{7}}
%\put(-387,0){$P_2$}
%\put(-394,0){\circle*{7}}
\put(-438,35){$P_{2,1}$}
\put(-417,32){\circle*{7}}
\put(-468,67){$P_{2,3}$}
\put(-447,63){\circle*{7}}
\put(-463,10){$S_{3,1}^-$}
\put(-446,2){\circle*{7}}
\put(-525,70){$S_{3,1}^+$}
\put(-510,55){\circle*{7}}
\put(-525,0){$R_3$}
\put(-508,2){\circle*{7}}
\put(-462,132){$S_{1,1}^+$}
\put(-448,124){\circle*{7}}
\end{picture}
\vspace{-3cm}
\caption{Graph $\Ga_2,\,m=3$}\label{Figure3}
\end{figure}
\newpage
\subsection{Two links are not diffeomorphic}
The calculation in the previous subsection shows the following important theorem.
Let $K_{f_i}:=V_{f_i}\cap S_\eps^{5}$ for $i=1,2$. By plumbing argument,
$K_{f_i}$ is diffeomorphic to  the boundary of the tubular neighborhood of exceptional divisors (\cite{Mum,Hirz}).
\begin{Theorem}\label{main2}
The first Betti numbers of the links $K_{f_1}$ and $K_{f_2}$ are given as
$3$ and $4$ respectively. In particular, $K_{f_1}$ and $K_{f_2}$ are not homeomorphic.
\end{Theorem} 
\begin{proof}
The Betti number of the link is  given as $2g_{tot}+r$ where $r$ is 
number of independent cycles in the graph  and $g_{tot}$ is the sum of the genera of the exceptional divisors (see (2),\cite{Hirz}). Now the assertion is immediate from the graph $\Ga_1,\Ga_2$ and the observation that  $P_{2,3}$ is the unique non-rational divisor with
$g(P_{2,3})=1$.
%and all other divisors are rational.
%$b_1(K_{f_2})=4$ as $Q$ has genus 1.
\end{proof}
\subsubsection{Wang sequence and Jordan block}
As far as we know, this is a first example of a pair of links with the same zeta function and non-diffeomorphic links.
Recall the Wang sequence of the Milnor fibration {\cite{Milnor}):
\[
0\to H_3(S_\eps^{5}\setminus K_{f_i})\to H_2(F_i)\mapright{h_*-\id} H_2(F_i)\to H_2(S_\eps^{5}\setminus K_{f_i})\to 0.
\]
Here $F_i$ is the Minor fiber of $f_i$.
%For simplicity, we consider the homology with $\mathbb Q$ coefficient.
By Alexander duality, $H^1(K_{f_i})\cong H_3(S_\eps^{5}\setminus K_{f_i})\cong \Ker\,(h_*-\id)$.
Theorem \ref{main2} says that the monodromy mappings for $f_1$ and $f_2$ have different Jordan
blocks on the second homology of the Milnor fiber, though their characteristic polynomials are 
given by 
\nl
$(1-t^4)^{3}(1-t^{4+m})^{4}(1-t)^{-1}$ and therefore
  the multiplicity of eigenvalue $1$ for $h_*: H_2(F)\to H_2(F)$ is 6 for both of $f_1,f_2$. 
  %More precisely, the multiplicity of 1 as eigenvalue is 6 for both of $f_1$ and $f_2$
  However  the number of Jordan blocks of eigenvalue 1 is 3 and 4 respectively for $f_1$ and $f_2$
  by Theorem \ref{main2}.
  % implies that the number of Jordan blocks is 3 and $4$ for $f_1$ and $f_2$ respectively.
  %%%%
  \subsection{Intersection numbers and dual resolution graphs}\label{intersection-number}
  To compute the self-intersection numbers, we consider the divisor of pull-back  function $\Pi^*x$. For simplicity, we consider the case $m=2m_0+1$ and $m_0\ge 1$.
  By (\ref{int1}) and the center of $\omega_i$ does not intersect with the coordinate plane $u_{\xi1}=0$, we get 
  \begin{eqnarray}
  (\Pi^*x)&=&(u_{\xi1})+P+\sum_{\al=1}^3 \sum_{i=1}^{m_0}S_{\al i}+2R_\al,\quad \text{for}\,\,f_1,\, f_2\label{int2}\\
  &=&(u_{\xi1})+P+\sum_{\al=1}^4 \sum_{i=1}^{m_0}S_{\al i}+2R_\al,\quad \text{for}\,\,f_3,\, f_4\label{int3}
  \end{eqnarray}
  Here $S_{\al i}=S_{\al i}^++S_{\al i}^-$,  and $P$ it is equal to $P_1$ for $f_1$ and $P_{2,3}+P_{2,1}$   for $f_2$ and  $P=P_{3,3}+P_{3,1}$ for $f_3$,  and $P=P_{4,2}+P_{4,2}'$ for $f_4$.
  Note that the genus of $P_{2,3}$ is 1 for $f_2$ but $P_{3,3}$ is a normalization of a nodal cubic $c_3^{ (1)}=0$ and it is  rational.
  $P_{4,2}, P_{4,2}'$ are smooth conics.
  % for $f_3$ as $Q$ was a nodal cubic in $f_3$.
  Using the property that 
 % intersection number of a compact exceptional divisor and 
  $(\Pi^*x)\cdot C=0$ for a compact exceptional divisor $C\subset Y$
   (see for example Theorem 2.6, \cite{La}), we get for $\Ga_1$ and $\Ga_2$:
  \begin{eqnarray*}
  &P_1^2=-10,\quad &{S_{\al i}^\pm}^2=-2,\quad R_\al^2=-1,\quad\text{for}\,\,f_1\\
  &P_{2,3}^2=-6, P_{2,1}^2=-4,& {S_{\al i}^\pm}^2=-2,\quad  R_\al^2=-1,\quad\text{for}\,\,f_2
  \end{eqnarray*}
  and for the resolution of $f_3, f_4$, exceptional divisors are all rational and:
  \begin{eqnarray*}
  P_{3,3}^2=-8&,\quad P_{3,1}^2=-4,\quad& {S_{\al i}^\pm}^2=-2,\quad  R_\al^2=-1,\quad\text{for}\,\,f_3\\
  P_{4,2}^2=-6&,\quad {P_{4,2}'}^2=-6,\quad &{S_{\al i}^\pm}^2=-2,\quad  R_\al^2=-1,\quad\text{for}\,\,f_4\\
  \end{eqnarray*}
  We have used the following  equality for the calculation:
  \begin{eqnarray}\label{int1}
 (u_{\xi 1})\cdot E(P)=4
 \end{eqnarray}
 as $E(P)\cap (u_{\xi 1})$ corresponds to the roots of $q_i(0,u_{\xi2},1)=0$ for each $i=1,\dots, 4$. 
\subsubsection{Remarks on the pair $\{f_3,f_4\}$}
The calculation of the links $K_{f_3}$ and $K_{f_4}$ are similar. We only give few remarks and leave the detail to the reader. Put $V_j=f_j\inv(0)$  for $j=3,4$. Let $\Ga_3,\Ga_4$ be the dual resolution graphs for $V_3$ and $V_4$ respectively.
We put for $V_3$,  $P_{3,3}\cap P_{3,1}=\{\rho_1,\rho_2,\rho_3\}$ and $\rho_4$ is the inner singularity of $P_{3,3}$.
$P_{3,3}$ is the normalization of the nodal cubic component and $P_{3,1}$ is the line component.
For $V_4$, $P_{4,2}\cap P_{4,2}'=\{\rho_1,\dots,\rho_4\}$
where $P_{4,2}, P_{4,2}'$ are conics.
\nl
(1) After the first blowing up $\hat \pi:X\to \mathbb C^3$, the exceptional divisor $ E(P)$ for $V_3$ is a union of cubic $P_{3,3}$ with one node and a line $P_{3,1}$ where  $P_{3,3}$ and $P_{3,1}$ intersects transversely. For $V_4$, $E(P)$ is two smooth conics $P_{4,2}$ and $P_{4,2}'$. They intersect transversely.
The modifications at 4$A_1$ singularities are exactly same as those in the previous section using the same regular simplicial cone subdivision $\Si^*$. For simplicity, we explain the outline assuming $m=2m_0+1$. We use the regular simplicial cone subdivision $\Si_o^*$  as before for the toric modification at each singular points.
\nl
(2) For three $A_1$ singularities $\rho_1,\rho_2, \rho_3$ of $E(P)$ are located at the intersection of two components
$P_{3,3}\cap P_{3,1}$ for $V_{f_3}$
or $P_{4,2}\cap P_{4,2}'$ for $V_{f_4}$, we take the toric modification $\omega_\al:Y_\al\to X$. Two divisors $P_{3,3}$ and $P_{3,1}$ or $P_{4,2}$ and $P_{4,2}'$ are separated by $\omega$ and  it gives  a tree of exceptional divisors  from $P_{3,3}$ to  $P_{3,1}$ or from $P_{4,2}$ to $P_{4,2}'$ respectively:
\[
 \overset {P_{3,3}}{\underset{P_{4,2}}{}} \bullet\rule[1mm]{.5cm}{0.3mm} \overset {S_{\al 1}^+} \bullet\rule[1mm]{.5cm}{0.3mm}  \dots\rule[1mm]{.5cm}{0.3mm} \overset {S_{\al m_0}^+} \bullet\rule[1mm]{.5cm}{0.3mm} \overset{R_\al} \bullet \rule[1mm]{.5cm}{0.3mm}\overset {S_{\al m_0}^- }\bullet \rule[1mm]{.5cm}{0.3mm} \dots\rule[1mm]{.5cm}{0.3mm} 
\rule[1mm]{.5cm}{0.3mm}   \overset {S_{\al1}^-} \bullet\rule[1mm]{.5cm}{0.3mm}
\bullet\overset{P_{3,1}}{\underset{P_{4,2}'} {}}
\]
%\nl
(3) The last $A_1$ singularity $\rho_4$  is an inner singularity of $P_{3,3}$ for $V_3$ and the fourth intersection of two conics for $V_4$. Thus after a toric modification $\omega_\al:Y_4\to X$,
$\Ga_3$ get a closed chain at $P_{3,3}$. For $\Ga_4$, it is a same pass as in (2).
The resolution graphs  are given in Figure \ref{Figure4}.
Two graphs have same number of independent cycles, 3.
All exceptional divisors are $\mathbb P^1$ and the number of independent cycles is 3 for $\Ga_3$ and $\Ga_4$. ($P_{3,3}$ is rational as it has one node.)
Thus $K_{f_3}$ and $K_{f_4}$ has  the first Betti number 3. However their 
graphs are not isomorphic (even as  weighted graphs).
To show that two links $K_{f_3}$ and $K_{f_4}$ are not diffeomorphic, we can use Theorem 3.2, \cite{Neumann}.
\begin{figure}[htb]  %ebb *.pdf
\setlength{\unitlength}{1bp}
\begin{picture}(600,300)(-0,-100)
{\includegraphics[width=12cm, bb=0 0 362 200]{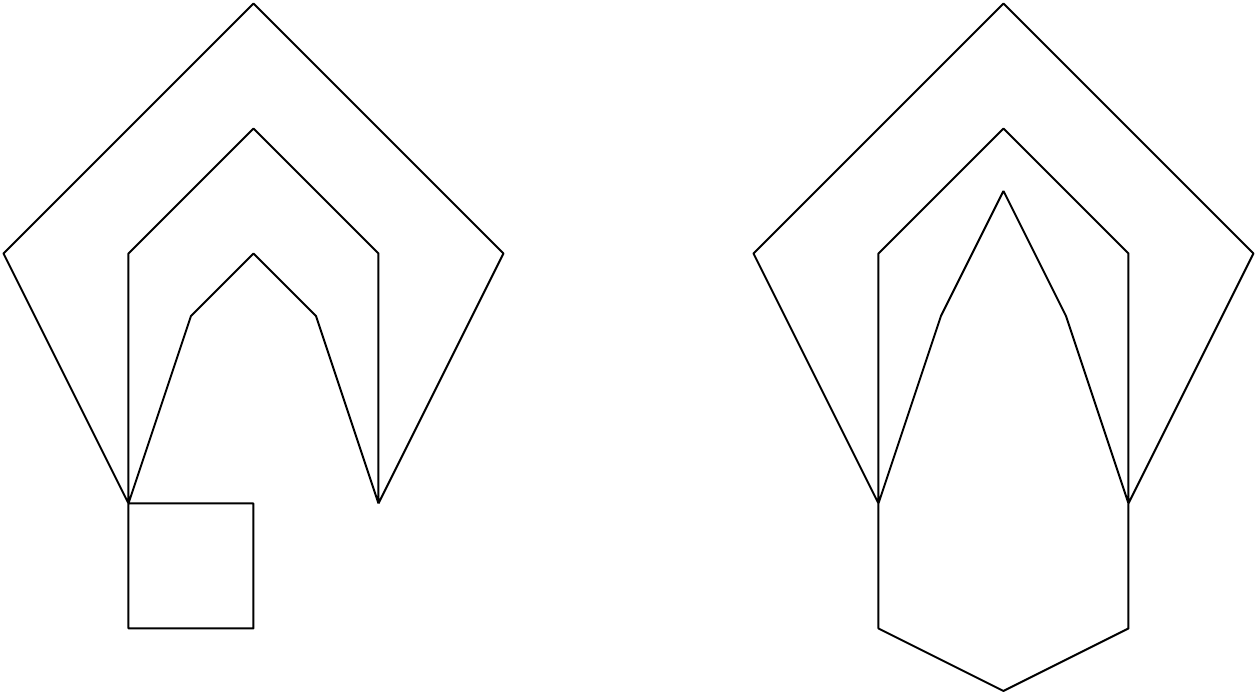}}
\put(-350,130){$S_{1,1}^+$}
\put(-340,120){\circle*{5}}
\put(-270,190){$R_1$}
\put(-270,186){\circle*{5}}
\put(-71,190){$R_1$}
\put(-68,187){\circle*{5}}
\put(-274,160){$R_2$}
\put(-271,153){\circle*{5}}
\put(-71,160){$R_2$}
\put(-69,152){\circle*{5}}
\put(-71,138){$R_3$}
\put(-68,135){\circle*{5}}
\put(-274,123){$R_3$}

\put(-272,118){\circle*{5}}
\put(-296,110){$S_{3,1}^+$}
\put(-288,102){\circle*{5}}
\put(-260,110){$S_{3,1}^-$}
\put(-254,102){\circle*{5}}
\put(-98,110){$S_{3,1}^+$}
\put(-86,100){\circle*{5}}
\put(-60,110){$S_{3,1}^-$}
\put(-52,102){\circle*{5}}
\put(-235,120){$S_{2,1}^-$}
\put(-238,119){\circle*{5}}
\put(-212,130){$S_{1,1}^-$}
\put(-205,120){\circle*{5}}
\put(-156,120){$S_{1,1}^+$}
\put(-136,120){\circle*{5}}
\put(-124,120){$S_{2,1}^+$}
\put(-102,120){\circle*{5}}
\put(-29,120){$S_{2,1}^-$}
\put(-36,120){\circle*{5}}
\put(0,120){$S_{1,1}^-$}
\put(-1,118){\circle*{5}}
\put(-320,127){$S_{2,1}^+$}
\put(-306,120){\circle*{5}}
\put(-326,52){$P_{3,3}$}
\put(-305,52){\circle*{5}}
\put(-330,22){$S_{4,1}^-$}
\put(-306,18){\circle*{5}}
\put(-230,52){$P_{3,1}$}
\put(-237,53){\circle*{5}}
\put(-282,60){$S_{4,1}^+$}
\put(-272,52){\circle*{5}}
\put(-270,20){$R_4$}
\put(-271,18){\circle*{5}}
\put(-124,50){$P_{4,2}$}
\put(-102,52){\circle*{5}}
\put(-60,50){$P_{4,2}'$}
\put(-35,52){\circle*{5}}
\put(-124,20){$S_{4,1}^+$}
\put(-102,18){\circle*{5}}
\put(-60,20){$S_{4,1}^-$}
\put(-34,19){\circle*{5}}
\put(-71,10){$R_4$}
\put(-69,2){\circle*{5}}
\put(-270,-25){$\Ga_3$}
\put(-71,-25){$\Ga_4$}
\end{picture}
\vspace{-3cm}
\caption{Graphs $\Ga_{3},\,\Ga_4$}\label{Figure4}
\end{figure}

%\newpage
\subsection{Link pairs constructed from Zariski pairs}
Consider a Zariski pair of projective curves 
$C: f_d(x,y,z)=0$ and $C': g_d(x,y,z)=0$ of degree $d$ with simple singularities. We assume thet $C,C'$ are irreducible for simplicity.
Consider the isolation 
\[\begin{split}
f(x,y,z):=f_d(x,y,z)+z^{d+m}\\
g(x,y,z):=g_d(x,y,z)+z^{d+m}.
\end{split}
\]
We assume $f$ and $g$ are convenient polynomial as before.
Take the simplest toric modification $\hat\pi:X\to \mathbb C^3$ and 
$\hat \pi':X'\to  \mathbb C^3$ with  $\Si^*$ with 4 vertices  $\{e_1,e_2,e_3,P\}$ with $P={}^t(1,1,1)$ as before. In the toric coordinates $\si=\Cone\,(e_1,e_2,P)$,
$\hat f$ and $\hat {g}$ are written as 
\[
\begin{split}
\hat f(\mathbf u_\si)&=u_{\si 3}^d(f_d(u_{\si 1},u_{\si 2},1)+u_{\si 3}^{m}),\\
\hat g(\mathbf u_\si)&=u_{\si 3}^d(g_d(u_{\si 1},u_{\si 2},1)+u_{\si 3}^{m})
\end{split}
\]

 Let $\rho_1,\dots, \rho_s$ be the singular points of $E(P)$. Then choose an admissible coordinates system  $\mathbf w_i=(w_{i,1},w_{i,2},w_{i,3})$ for each $i=1,\dots, s$
 with $w_{i,3}=u_{\si 3}$. As $E(P)$ is projective space $\mathbb P^2$ and 
 $f_d(u_{\si 1},u_{\si 2},1)=0$ and $g_d(u_{\si 1},u_{\si 2},1)=0$ is the affine equation of the projective curve $C, C'$ respectively,
 we may assume that $\hat f$ and $\hat g$ take the exact same polynomial expression at each $\rho_i$.
 That is, 
 $\hat f, \hat g$ take the form
 \[
 \hat f(\mathbf w), \hat g(\mathbf w)=u_{\si 3}^d\left(\psi_i(w_1,w_2)+
 u_{\si 3}^{m}
 \right)
 \]
 where $\psi_i(w_1,w_2)$ is  a fixed normal form of the simple singularity $(E(P),\rho_i)=(C,\rho_i)$ at $\rho_i$.
  We proceed  the further  toric modifications
 at each $\rho_i$, $\omega_i: Y_i\to X$ or $\omega_i': Y_i'\to X'$  with respect to the same regular simplicial cone subdivision $\Si_i^*$.
 Let $\Pi: Y\to \mathbb C^3$  and $\Pi': Y'\to \mathbb C^3$ be the resolution of $f$ and $g$ obtained by composing these toric modification with $\hat \pi$ as before.
 In this way two surface singularities get the exact same configuration of the exceptional divisors $E_{i,1},\dots, E_{i,r_i}$. Here we are abusing notations $E_{i,j}$'s which are  exceptional divisors of either  $\Pi$ or of $\Pi'$. Thus two hypersurface $V(f)$ and $V(g)$ have the exact same dual
 resolution graph. The assumption that $\{C,C'\}$ is a Zariski pair of irreducible curves implies after the resolution $\Pi$ and $\Pi'$, 
 the central divisor $E(P)$ has the same genus for $f$ and $g$. (This was not the case for a weak Zariski pair.)
 We can compute intersection numbers of exceptional divisors using  the divisor
 $(\Pi^*x)$ as in \S \ref{intersection-number} and use the property that $(\Pi^*x)\cdot E=0$ for any compact divisor $E$ in $Y$. 
 %See for example, Theorem 2.6,\cite{La}.
 It is easy to see that  $(\Pi^*x)$ has the exactly same expression  for $V(f)$ and $V(g)$.
 Thus as the link 3-manifolds $K_f$ and $K_{g}$ can be considered as  the graph manifolds, we have
 \begin{Theorem} \label{main3}
 Assume that $\{C,C'\}$ is a Zariski pair of irreducible curves with simple singularities.  The two links $K_f$ and $K_{g}$  are diffeomorphic.
 \end{Theorem}
In \cite{Almost}, we gave an example of such links. Though two links are diffeomorphic, we do not know if the diffeomorphism can be extended to a diffeomorphism of $S_\eps^{5}$ or not.
\begin{Remark}{\rm
Theorem \ref{main3} also valid for non-irreducible Zariski pairs. The argument is exactly same. The simple singularities assumption  can be replaced by  Newton no-degeneracy of  singularities.
The assumption that $\{C,C'\}$ is a Zariski pair is crucial, because otherwise, the geometry of the central divisor $E(P)$ in $Y$ and $Y'$ are different as we have seen in the case of weak Zariski pair.}
\end{Remark}

\subsection{Appendix}
\label{Appendix}
Starting from weak Zariski pairs, we can construct many other examples with non-diffeomorphic  links. We give two examples. We leave the detail for the reader. More interesting  problem is: {\em Are they 
$\mu$-Zariski pairs of links ?}

Consider irreducible projective curve of degree $d$ with $k$ $A_1$'s singularities and note it as $C_d^{(k)}$.
Note that the genus of the normalization of $C_d^{(k)}$  is $\frac{(d-1)(d-2)}2-k$.
We denote by $r$ the number of independent cycles in the resolution graph and by $g_{tot}$  the sum of the genus of the exceptional divisors. Put 
$b_1=r+2g_{tot}$,  which is the Betti number of the link. The calculation can be done in exact same say as our examples $f_1,f_2,f_3, f_4$.

Example 1. Consider the pair of sextic curves $\{C_{6,1}, C_{6,2}\}$ with $9$ $A_1$ singularities where $C_{6,1}=C_6^{(9)}$ (9 nodal sextic) and $C_{6,2}=C_3^{(0)}+C_3^{(0)}$, two generic cubics. We assume that irreducible components are intersecting transversely.
Let $f_{6,i}(x,y,z)$ be the defining homogeneous polynomials. We always assume that $f_{6,i}$ is convenient.
We consider the isolation surfaces
\[
V_i: g_{i}(x,y,z):=f_{6,i}(x,y,z)+z^{6+m}=0,\, i=1,2, \, m\ge 1.
\]
 The resolution is given exactly as 
\S \ref{ResolutionV1V2}. After one point blowing up $\hat \pi:X\to \mathbb C^3$, $E(P)$ has 9 nodes and they are resolved by 9 toric modifications
omega: $Y\to X$. At each node of $E(P)$, the second toric resolution is exactly as in \S \ref{ResolutionV1V2}. In $V_1$, $E(P)$ is normalized by $\omega$ into a genus 1 curve. For $V_2$, $E(P)$ has two torus (=surface of genus 1) and separated by $\omega$.
Then the corresponding links $K_{i}$  are not diffeomorphic and the invariants of the dual resolution graphs are given as follows.
\begin{table}[hbtp]
\caption{Invariants}
\centering
\begin{tabular}{lccc}
\hline
 link&{$ r$}&{$g_{tot}$}&{$b_1$}\\
 \hline\hline
 {$K_1$}&{$9$}&{$1$}&{$11$}\\
 \hline
 {$ K_2$}&{$8$}&{$2$}&{$12$}\\
 \hline
\end{tabular}
\end{table}
Their zeta function is given by $(1-t^6)^{-12}(1-t^{6+m})^{-9}$ and 
$\mu^*=(134,25,5)$.

Example 2. Consider triple of projective curve
$\{C_{6,3},C_{6,4},C_{6,5}\}$ of sextics with 10 $A_1$ singularities
where $C_{6,3}=C_6^{(10)}$ (10 nodal sextic), $C_{6,4}=C_3^{(1)}+C_3^{(0)}$
and $C_{6,5}=C_4^{(1)}+C_1+C_1'$ with $C_1, C_1'$ being lines. Irreducible components are intersecting transversally.
Let $f_{6,j}(x,y,z)$ be the defining convenient homogeneous polynomials
and let $g_{6,j}(x,y,z)=f_{6,j}(x,y,z)+z^{6+m},\,j=3,4,5$ respectively.
Let $K_j, j=3,4,5$ be the corresponding link 3-manifolds. Among the exceptional divisors, $C_3^{(0)}$ has genus 1 and 
the normalization of $C_4^{(1)}$ has genus 2 and the corresponding invariants of the resolution graphs are given as follows.

\begin{table}[hbtp]
\caption{Invariants}
\centering
\begin{tabular}{lcccc}
\hline
 link&{$ r$}&{$g_{tot}$}&{$b_1$}\\
 \hline\hline
 {$K_3$}&{$10$}&{$0$}&{$10$}\\
 \hline
 {$ K_4$}&{$9$}&{$1$}&{$11$}\\
 \hline
 $K_5$&$8$  & $2$ & $12$\\
 \hline
\end{tabular}
\end{table}
\noindent
Their zeta-function is $(1-t^6)^{-11}(1- t^{6+m})^{-10}$. The $\mu^*$-invariant is 
$(135,25,5)$.

\vspace{.3cm}
{\em A related  problem.} 1. In the proof of Theorem \ref{mu-zariski}, we have proved that the multiplicity of $\mu$-constant family is   constant in our example $\{f_1,f_2\}$ or $\{f_3,f_4\}$. We ask if this  is true for other $\mu$-constant family.
If this is not true, give an explicit counter example.

2. Suppose that $\{g_1,g_2\}$ is a weak Zariski pair (respectively Zariski pair) of degree $d$ and consider $g_i:=f_i+z^{d+m},\,i=1,2$. 
\begin{enumerate}
\item If $\{f_1,f_2\}$ is a weak Zariski pair,  is $\{f_1,f_2\}$ is a $\mu$-Zariski pair of hypersurface?
\item If they are a weak Zariski pair, but not a Zariski pair, are their links $K_{g_i}, i=1,2$ not  diffeomorphic?

\item Suppose $\{f_1,f_2\}$ is a Zariski pair. Is  $\{g_1,g_2\}$ a $\mu$-Zariski pair? (We have shown that they are $\mu^*$-Zariski pair in \cite{EO}.
\end{enumerate}
3. Are the examples in this appendix  $\mu$-Zariski pairs (resp. $\mu$-Zariski triple)?
 % We leave this work to the reader.
%%%%%%
%\newpage
%$\Cone(P
%In Figure \ref{Figure2}, we show the most economical regular simplicial subdivision for $m=1$ and $m=2$.
 %%%%%%%%%%%%%%%%%%%%%%%
%\newpage
\def\cprime{$'$} \def\cprime{$'$} \def\cprime{$'$} \def\cprime{$'$}
  \def\cprime{$'$} \def\cprime{$'$} \def\cprime{$'$} \def\cprime{$'$}

%\bibliographystyle{abbrv}
%\bibliography{/home/oka/paper/okabib}
\end{document}